\documentclass
{article}
\usepackage[latin1]{inputenc}
\usepackage{indentfirst}

\usepackage{amsmath,amsfonts,amssymb,amsthm,amscd}
\RequirePackage{ifthen}
\RequirePackage{hyperref}
\usepackage{latexsym}
\usepackage{mathrsfs}
\usepackage{algorithm}
\usepackage[noend]{algorithmic}
\usepackage{paralist}
\usepackage{graphicx}
\usepackage{booktabs}
\usepackage{mathtools}
\usepackage{rotating}
\usepackage{tikz}

\usetikzlibrary{decorations.shapes}

\usepackage{comment}

\usepackage{subcaption}

\usepackage[normalem]{ulem}
\usepackage[color,matrix,arrow]{xy}
\usepackage{mathtools}
\usepackage{xfrac}

\usetikzlibrary{patterns}

\pgfdeclarepatternformonly{my crosshatch dots}{\pgfqpoint{-1pt}{-1pt}}{\pgfqpoint{9pt}{32pt}}{\pgfqpoint{18pt}{32pt}}%
{
    \pgfpathcircle{\pgfqpoint{0pt}{0pt}}{1pt}
    \pgfpathcircle{\pgfqpoint{9pt}{15.588pt}}{1pt}
    \pgfusepath{fill}
}

\usepackage[a4paper,margin=1in]{geometry}

\input xy
\xyoption{all}
\newcommand{\leqdr}{\mathbin{\rotatebox[origin=c]{-45}{$\leq$}}}
\newcommand{\lequr}{\mathbin{\rotatebox[origin=c]{45}{$\leq$}}}
\newdir{d}{\leq}
\newdir{d}{\geq}
\newdir{dr}{\leqdr}
\newdir{ur}{\lequr}

\theoremstyle{plain}

\theoremstyle
{plain}
\newtheorem{theorem}{Theorem}[section]

\newtheorem{proposition}[theorem]{Proposition}

\newtheorem{lemma}[theorem]{Lemma}
\newtheorem{corollary}[theorem]{Corollary}

\newtheorem{question}[theorem]{Question}
\newtheorem{claim}[theorem]{Claim}
\theoremstyle{definition}
\newtheorem{definition}[theorem]{Definition}

\newtheorem{remark}[theorem]{Remark}

\makeindex

\newcommand{\N}{\mathbb{N}}
\newcommand{\Z}{\mathbb{Z}}
\newcommand{\Q}{\mathbb{Q}}
\newcommand{\R}{\mathbb{R}}

\newcommand{\QQ}{\mathcal Q}

\DeclareMathOperator{\diam}{diam}

\DeclareMathOperator{\PPS}{PPS}

\DeclareMathOperator{\Iso}{Isom}
\DeclareMathOperator{\den}{den}
\DeclareMathOperator{\Vol}{Vol}
\DeclareMathOperator{\Latt}{Latt}
\DeclareMathOperator{\dist}{dist}

\newcommand{\EH}{d_{EH}}
\newcommand{\GH}{d_{GH}}
\newcommand{\GB}{d_{GB}}
\newcommand{\EB}{d_{EB}}

\newcommand{\Addresses}{{
		\bigskip
		\footnotesize
		
		\noindent A.~Garber \\ \textsc{University of Texas Rio Grande Valley, Brownsville, TX, USA}\\
		\textit{E-mail address}: \texttt{alexey.garber@utrgv.edu}\\

		\medskip
  
        \noindent \v{Z}.~Virk, \textsc{University of Ljubljana and institute IMFM, Ljubljana, Slovenia}\\
		\textit{E-mail address}: \texttt{ziga.virk@fri.uni-lj.si}\\
		
		\medskip
		
		\noindent N.~Zava,  \textsc{Institute of Science and Technology Austria (ISTA), Klosterneuburg, Austria}\\
		\textit{E-mail address}: \texttt{nicolo.zava@gmail.com}
}}

\author{Alexey Garber, \v{Z}iga Virk, Nicol\`o Zava}
	\title{On the metric spaces of lattices and periodic point sets}

\begin{document}
\maketitle

\tableofcontents

\begin{abstract}
    Lattices and periodic point sets are well known objects from discrete geometry. They are also used in crystallography as one of the models of atomic structure of periodic crystals. In this paper we study the embedding properties of spaces of lattices and periodic point sets equipped with optimal bijection metrics (i.e., bottleneck and Euclidean bottleneck metrics). We focus our treatment on embeddings into Hilbert space. On one hand this is motivated by modern data analysis, which is mostly based on statistical approaches developed on Euclidean on Hilbert spaces, hence such embeddings play a major role in applied pipelines. On the other hand there is a well-established methodology related to such questions in coarse geometry, arising from the work on the Novikov conjecture. The main results of this paper provide different conditions, under which the spaces of lattices or periodic point sets are Lipschitz or coarsely (non)embeddable into Hilbert space. The various conditions are phrased in terms of density, packing radius, covering radius, the cardinality of the motif, and the diameter of the unit cell. 
\end{abstract}



\section{Introduction}

One of the fundamental problems in mathematics is that of embedding: given a collection of objects, can we embed it into a more manageable space? Such embeddings often reveal the intricate nature of the objects and lead to further applications. Some of the more prominent examples of this line of thought are the planarity of graphs, embedding dimensions of manifolds, representation theory, and partial solutions to the Novikov conjecture via coarse embeddings. Recent advances in data analysis have led to a particular interest in embeddings into Hilbert space, where tools of statistics can be readily applied. 

In this paper, we focus on embedding problems for the family of lattices and periodic point sets equipped with a certain metric. A lattice in $\mathbb{R}^d$ is a linearly transformed standard integer lattice $\mathbb{Z}^d \subset \mathbb{R}^d$, while a periodic point set is the Minkowski sum of a lattice and a finite set called motif. Both concepts feature prominently in discrete geometry and crystallography as they can be treated as an orbit or a union of finitely many orbits of a crystallographic group, i.e. a group of isometries of $\R^d$ with compact fundamental domain. 

The systematic study of crystallographic or space groups and their orbits goes back to works of Barlow, Fedorov and Sch\"onflies who enumerated all 230 of them in $\R^3$, see \cite{CDHT01}. The structural questions of crystallographic groups were a part of Hilbert's 18th problem which asked whether every $d$-dimensional crystallographic group contains a $d$-dimensional subgroup of translations of finite index. This question was positively resolved by Bieberbach \cite{Bie1,Bie2}.

The discovery of x-ray crystallography confirmed that the atoms crystals are usually arranged as periodic point sets, see \cite{Sen}. However, the crystallization process of self-organization of atoms still remains as one of the fundamental problems in crystallography \cite{Rad}.

In our approach, we adopt the following technical setup. 
The metrics used on the spaces of lattices and periodic point sets are the bottleneck and the Euclidean bottleneck distances (see Definition \ref{DefDB}), which measure ``optimal'' bijections between sets. We are interested in bi-Lipschitz and coarse embeddings (see Subsection \ref{sec:results} for the definitions) into Hilbert space. 

The \textbf{main results} of this paper are the following. 

\begin{enumerate}
 \item \emph{Lipschitz non-embeddability}. The space of lattices in $\mathbb{R}^d$ of prescribed density $\kappa$ is not separable even if we bound the packing radius from below and covering radius from above (Theorem \ref{theo:uncountable_family}). Hence it is not Lipschitz embeddable into a separable Hilbert space, and neither is the space of similarly restricted periodic points sets.
 \item \emph{Coarse non-embeddability}. The space of periodic point sets in $\mathbb{R}^d$ of prescribed density $\kappa$ does  not coarsely embed into Hilbert space even if we bound the packing radius from below, or covering radius from above
 (Theorem \ref{theo:no_coarse_embedding}).
 \item \emph{Coarse embeddability}. The space of periodic point sets in $\mathbb{R}^d$ of prescribed density $\kappa$ does coarsely embed into Hilbert space if we bound the covering radius and the cardinality of motif from above  (Theorem \ref{theo:bounded_motif}), or if we bound the diameter of the fundamental cell from above (Theorem \ref{theo:bounded_cell}).
\end{enumerate}

The main results and their consequences are summarized in Table \ref{table:embeddability_results}. Throughout the paper we also prove a number of secondary technical results, which should be of independent interest. These include the local geometric structure of periodic point sets in Section \ref{sect:BMS} and the existence of minimizing isometry for bottleneck matching in Appendix. 

\textbf{Related work.}
The bottleneck distance, or {\it bounded distance equivalence}, for various representatives of the family of the periodic sets and for a wider family of Delone sets (see Figure \ref{fig:p_and_c} below), including quasiperiodic sets, has been studied before. Particularly, the bijection witnessing the bottleneck distance between two Delone sets can be treated as a bounded transport between discrete measures originating from the sets.

Let us mention a few results concerning the periodic point sets. Duneau and Oguey \cite{DunOgu} proved that every two lattices of the same density are at finite bottleneck distance; the property also extends to two periodic point sets of equal density. Laczkovich \cite{Lac} established a criterion when an arbitrary (discrete) point set in $\R^d$ is at finite bottleneck distance from a lattice. We also note that for hyperbolic spaces $\mathbb H^d$, every two Delone sets are at finite bottleneck distance due to Bogopolski \cite{Bog}.

Additionally we mention several works that study bottleneck distance between lattices and (weighted or unweighted) quasiperiodic Delone sets obtained by cut-and-project scheme \cite{HayKoi, HKK, Haynes, FG-weighted} and by substitutions techniques \cite{Sol,FG-pisot}. The former setting resembles connections with Diophantine approximations, equidistributions of fractional parts of $\alpha n$ where $\alpha\in \R\setminus \Q$, and bounded remainder sets. 

The bottleneck distance for periodic point sets has been also used in \cite{EdeHeiKurSmiWin} as an application to develop a density fingerprint that helps to identify isometric periodic point sets.

The results mentioned above concern some distinguished representatives of the metric space of the periodic point sets while we are concentrating on the global structure of that space and its relation to coarse geometry.

Coarse geometry arose from the study of asymptotic properties of groups and was to a large degree motivated by the work of Gromov \cite{Grom}. Coarse embeddings gained prominence as their role in the Novikov conjecture became apparent \cite{Yu}, see also \cite{DraGonLafYu, Roe}. In the context of topological data analysis there have been many attempts to embed persistence diagrams into Hilbert space. It was proved that there are no bi-Lipschitz embeddings \cite{CarBau}, see also \cite{BubWag}. On the other hand, coarse embeddability of persistence diagrams on $n$ points was established in \cite{MitVir}. This result was later refined in \cite{BatGar}, where the authors provide a concrete bi-Lipschitz embedding. On a similar note, certain non-embeddability results for hyperspaces were proved in \cite{WeiYamZav} and \cite{Za_GH}.

\textbf{Structure of the paper.} In Section \ref{sec2} we provide preliminaries. The main results are stated in a structured way in Subsection \ref{sec:results}. Boundedness results and consequent embeddability results are provided in Section \ref{sect:BMS}. In Section \ref{sec:bi-Lipschitz} we provide non-separability results and consequence bi-Lipschitz non-embeddability results. In Section \ref{sec:coarse} we prove non-embeddability results in the context of coarse geometry. A proof of the existence of minimizing isometries of the bottleneck and Euclidean bottleneck distances is provided in Appendix \ref{appendix}. 

{\bf Acknowledgements.} A.G. was partially supported by the Alexander von Humboldt Foundation. \v{Z}.V. was supported by Slovenian Research Agency grants No. N1-0114, J1-4001, and P1-0292. N.Z. was supported by the FWF Grant, Project number I4245-N35.

\subsection*{Notation}

We denote by $e_1,\dots,e_d$ the vectors of the standard basis of $\R^d$.
For a set $X$ and a natural number $n$,  we denote by $[X]^{\leq n}$  the family of subsets of $X$ with cardinality at most $n$. Furthermore, we define
$$[X]^{=n}=[X]^{\leq n}\setminus [X]^{\leq n-1},\quad\text{and}\quad [X]^{<\omega}=\bigcup_{n\in\N}[X]^{\leq n}.$$

\section{Definitions, basic properties, and statements of main results}\label{sec2}

\subsection{Lattices and periodic point sets}

A subgroup $\Lambda$ of the additive group $\R^d$ is a {\em lattice} if there exist $d$ linearly independent  vectors $v_1,\dots,v_d\in\Lambda$, called {\em generators}, such that every element $\lambda\in\Lambda$ can be written as an integer combination of those $d$ vectors. In other words,
$$\Lambda =\Z (v_1,\dots,v_d)=\bigg\{\sum_{i=1}^da_iv_i\mid a_i\in \Z, i=1,\dots,d\bigg\}$$
where $v_1,\dots,v_d$ is a basis of $\R^d$. In the previous notation, we define the {\em unit cell} of $\Lambda$ as the subset
$$U(\Lambda)=\bigg\{\sum_{i=1}^dt_iv_i\mid\,t_i\in[0,1[, i=1,\dots,d\bigg\}.$$
Its volume $\Vol(U(\Lambda))$, which is equal to the  (absolute value of) determinant\footnote{In order to avoid possible confusion, we assume that we always take absolute values of determinants when we write about lattices.} of the matrix $(v_1,\dots,v_d)$ that, for the sake of simplicity, we denote by $\det(\Lambda)$, is independent of the choice  of the generators. Let us denote by $\Latt(\R^d)$ the family of  all lattices of $\R^d$.

Let $X\subseteq\R^d$. We define:
\begin{compactenum}[(a)]
\item its {\em packing radius} as 
$$p(X)=\sup\{r\in\R_{\geq 0}\mid\forall x\in\R^d,\,\lvert B(x,r)\cap X\rvert\leq 1 \}=\inf_{x,y\in X, x\neq y}\frac{\lvert\lvert x-y\rvert\rvert}{2};$$
\item its {\em covering radius} as
$$c(X)=\inf\{R\in\R_{\geq 0}\mid\forall x\in\R^d,\,\lvert \overline{B(x,R)}\cap X\rvert\geq 1 \}=\sup_{y\in\R^d}\inf_{x\in X}\lvert\lvert x-y\rvert\rvert.$$
\end{compactenum}
Here $B(\cdot,a)$ and $\overline{B(\cdot,a)}$ are open and closed balls of radius $a$, respectively.

Note that, $p(X)\geq r$ if $\lvert\lvert x-y\rvert\rvert\geq 2r$ for every $x,y\in X$. If $X$ has strictly positive packing radius, then $X$ is countable. If its covering radius is  finite, then $X$ is infinite. A {\em Delone set}\footnote{Boris Delone (Delaunay), a Russian and Soviet mathematician of French descent. He used the French spelling Delaunay in earlier works and the transliteration of Russian spelling Delone in later works. We use
the latter spelling in the paper but the actual spelling in the cited references might be different.} is a subset $X$ of $\R^d$ with $p(X)>0$ and $c(X)<\infty$, and thus it is countably infinite. We illustrate the concepts of packing and covering radii in Figure \ref{fig:p_and_c}.

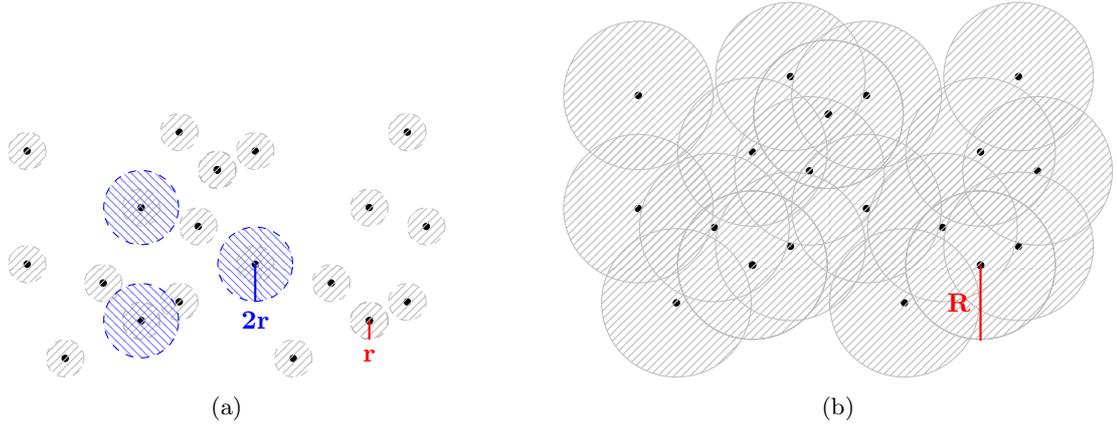
\begin{figure}
\begin{subfigure}[b]{0.5\textwidth}
\centering
\begin{tikzpicture}[>=latex]
			\fill (1,1) circle (1.4pt) (1.5,0.5) circle (1.4pt) (1.5,0.5) circle (1.4pt) (2,0.75) circle (1.4pt) (0.5,0) circle (1.4pt) (2.25,1.75) circle (1.4pt) (3,1.25) circle (1.4pt);
			\fill (4,1) circle (1.4pt) (4.5,0.5) circle (1.4pt) (4.5,0.5) circle (1.4pt) (5,0.75) circle (1.4pt) (3.5,0) circle (1.4pt) (5.25,1.75) circle (1.4pt); 
			\fill (0,1.25) circle (1.4pt); 
			\fill (2,3) circle (1.4pt) (2.5,2.5) circle (1.4pt) (2.5,2.5) circle (1.4pt) (3,2.75) circle (1.4pt) (1.5,2) circle (1.4pt);
			\fill (5,3) circle (1.4pt) (4.5,2) circle (1.4pt); 
			\fill (0,2.75) circle (1.4pt);
	
			\fill [draw=lightgray, dashed, pattern=north east lines, pattern color=lightgray] (1,1) circle (7pt) (1.5,0.5) circle (7pt) (1.5,0.5) circle (7pt) (2,0.75) circle (7pt) (0.5,0) circle (7pt) (2.25,1.75) circle (7pt) (3,1.25) circle (7pt); 
			\fill [draw=lightgray, dashed, pattern=north east lines, pattern color=lightgray] (4,1) circle (7pt) (4.5,0.5) circle (7pt) (4.5,0.5) circle (7pt) (5,0.75) circle (7pt) (3.5,0) circle (7pt) (5.25,1.75) circle (7pt);
			\fill [draw=lightgray, dashed, pattern=north east lines, pattern color=lightgray] (2,3) circle (7pt) (2.5,2.5) circle (7pt) (2.5,2.5) circle (7pt) (3,2.75) circle (7pt) (1.5,2) circle (7pt); 
			\fill [draw=lightgray, dashed, pattern=north east lines, pattern color=lightgray] (5,3) circle (7pt) (4.5,2) circle (7pt); 
			\fill [draw=lightgray, dashed, pattern=north east lines, pattern color=lightgray] (0,2.75) circle (7pt);
			\fill [draw=lightgray, dashed, pattern=north east lines, pattern color=lightgray] (0,1.25) circle (7pt);

			\draw[red, thick] (4.5,0.5)--(4.5,0.25) node [pos=1, below] {${\bf r}$};

\fill [draw=blue, dashed, pattern=north west lines, pattern color=blue!50!white] (1.5,2) circle (14pt);

\fill [draw=blue, dashed, pattern=north west lines, pattern color=blue!50!white] (1.5,0.5) circle (14pt);

\fill [draw=blue, dashed, pattern=north west lines, pattern color=blue!50!white] (3,1.25) circle (14pt);

			\draw[blue, thick] (3,1.25) -- (3,0.75) node [pos=1, below] {${\bf 2r}$};

		\end{tikzpicture}
	
 \caption{}
		\end{subfigure}
		\begin{subfigure}[b]{0.5\textwidth}
\centering    
	\begin{tikzpicture}[>=latex]
						\fill (1,1) circle (1.4pt) (1.5,0.5) circle (1.4pt) (1.5,0.5) circle (1.4pt) (2,0.75) circle (1.4pt) (0.5,0) circle (1.4pt) (2.25,1.75) circle (1.4pt) (3,1.25) circle (1.4pt);
			\fill (4,1) circle (1.4pt) (4.5,0.5) circle (1.4pt) (4.5,0.5) circle (1.4pt) (5,0.75) circle (1.4pt) (3.5,0) circle (1.4pt) (5.25,1.75) circle (1.4pt); 
			\fill (0,1.25) circle (1.4pt); 
			\fill (2,3) circle (1.4pt) (2.5,2.5) circle (1.4pt) (2.5,2.5) circle (1.4pt) (3,2.75) circle (1.4pt) (1.5,2) circle (1.4pt);
			\fill (5,3) circle (1.4pt) (4.5,2) circle (1.4pt); 
			\fill (0,2.75) circle (1.4pt);
	
			\fill [draw=lightgray, pattern=north east lines, pattern color=lightgray] (1,1) circle (28pt) (1.5,0.5) circle (28pt) (1.5,0.5) circle (28pt) (2,0.75) circle (28pt) (0.5,0) circle (28pt) (2.25,1.75) circle (28pt) (3,1.25) circle (28pt); 
			\fill [draw=lightgray, pattern=north east lines, pattern color=lightgray] (4,1) circle (28pt) (4.5,0.5) circle (28pt) (4.5,0.5) circle (28pt) (5,0.75) circle (28pt) (3.5,0) circle (28pt) (5.25,1.75) circle (28pt);
			\fill [draw=lightgray, pattern=north east lines, pattern color=lightgray] (2,3) circle (28pt) (2.5,2.5) circle (28pt) (2.5,2.5) circle (28pt) (3,2.75) circle (28pt) (1.5,2) circle (28pt); 
			\fill [draw=lightgray, pattern=north east lines, pattern color=lightgray] (5,3) circle (28pt) (4.5,2) circle (28pt); 
			\fill [draw=lightgray, pattern=north east lines, pattern color=lightgray] (0,2.75) circle (28pt);
			\fill [draw=lightgray, pattern=north east lines, pattern color=lightgray] (0,1.25) circle (28pt); 
			
			\draw[red, thick] (4.5,0.5)--(4.5,-0.5) node [pos=0.5, left] {${\bf R}$};
		\end{tikzpicture}
	
 \caption{}
\end{subfigure}
\caption{A portion of an infinite subset $X$ of $\R^2$ with $p(X)\geq r$ and $c(X)\leq R$. In both pictures (a) and (b), the black dots represent the point set. In (a), the gray balls of radius $r$ are centered on the black dots and do not overlap. Whereas, the blue balls have radius $2r$, are centered on the black dots, and contain no other black dot. In (b), the gray balls with radius $R$ centered on the black dots cover the entire space. 
}\label{fig:p_and_c}
\end{figure}

A particular class of Delone sets is given by {\em periodic point sets} of $\R^d$, which are subsets $X$ given by the Minkowski sum of a lattice $\Lambda\leq\R^d$ and a finite  subset $P\subseteq  U(\Lambda)$ called {\em motif}. More explicitly, $X$ is equal to $\Lambda+P=\{\lambda+x\mid\lambda\in\Lambda,x\in P\}$. 

Note that given a periodic point set $X$, there are many ways to choose a lattice $\Lambda$ and associated motif $P$ such that $X=\Lambda+P$. In each such situation, we refer to the pair $\Lambda$ and $P$ as {\em representatives} of $X$. 

To every periodic point set $\Lambda+P$ we associate its {\em density}, which is given by $\den(X)=\lvert P\rvert/\det(\Lambda)=\lvert P\rvert/\Vol(U(\Lambda))$.  
Even though for its definition we use a particular choice of representatives, the density of a periodic point $X$ set does not depend on it as it can be treated as the average number of points of the set per unit volume. Here the average is computed as 
$$\den(X)=\lim_{R\rightarrow \infty}\frac{\lvert X\cap \overline{B(0,R)}\rvert}{\Vol{\overline{B(0,R)}}}$$
using the number of points of $X$ in a large ball centered at the origin. Moreover, this definition agrees with a more general definition of density through van Hove sequences which captures a similar quantity (and the same value) for a broader class of Delone sets, see \cite[Sect. 2]{fre22} for more details.

Let us denote by $\PPS(\R^d)$ the family of all periodic point  sets  of $\R^d$. Moreover, we consider the following subsets of $\PPS(\R^d)$ restricted to subfamilies with bounds on density, packing or covering radii, or their combinations:
\begin{gather*}
\PPS_r(\R^d)=\{X\in\PPS(\R^d)\mid  p(X)>r\},\quad\PPS^R(\R^d)=\{X\in\PPS(\R^d)\mid c(X)<R\},\\
\PPS_r^R(\R^d)=\PPS_r(\R^d)\cap\PPS^R(\R^d),\quad\PPS(\R^d,\kappa)=\{X\in\PPS(\R^d)\mid\den(X)=\kappa\},\\
\PPS_r(\R^d,\kappa)=\PPS_r(\R^d)\cap\PPS(\R^d,\kappa),\quad\PPS^R(\R^d,\kappa)=\PPS^R(\R^d)\cap\PPS(\R^d,\kappa),\quad\text{and}\\
\PPS_r^R(\R^d,\kappa)=\PPS_r^R(\R^d)\cap\PPS(\R^d,\kappa).
\end{gather*}
We adopt a similar notation for subsets of $\Latt(\R^d)$. 

The packing and the covering radii together with the dimension and the density are interdependent parameters. One of the most famous examples of such dependency is the question of finding a $d$-dimensional point set that provides the best possible density of packing equal balls.

This question for $d=3$ is known as {\it Kepler's conjecture} and was part of Hilbert's 18th problem. Overall, the problem is solved only for $d=1,2,3,8$ and $24$ \cite{Hal05,Via17,CKMRV17} and in all these dimensions the set that gives the densest packing is a lattice.

The same question for lattices is also solved for $d=4,5,6$ and $7$. It is also worth noting that for some dimensions $d\geq 10$, the best known point sets are periodic but not lattices. We refer to \cite{Sch08} for an extensive survey on the questions of packings and coverings related to lattices and periodic point sets.

We also mention that while it is not known that for every dimension $d$ the densest packing with density $\delta_d$ is provided by a periodic point set, or even whether the best density is minimum and not infimum, it is possible to find a $d$-dimensional periodic point set that gives density at least $\delta_d-\varepsilon$ for every $\varepsilon>0$, see \cite{Sch08,ManMar22}.

\subsection{Extended pseudo-metrics on the space of periodic point sets}

Let $X$ be a set. An {\em extended pseudo-metric $d$} on $X$ is a map $d\colon X\times X\to\R_{\geq 0}\cup\{\infty\}$ satisfying the following properties:
\begin{compactenum}[(a)]
	\item for every $x\in X$, $d(x,x)=0$;
	\item for every $x,y\in X$, $d(x,y)=d(y,x)$;
	\item for every $x,y,z\in X$, $d(x,z)\leq d(x,y)+d(y,z)$ with the convention that, for every $a\in\R$, $a<a+\infty=\infty=\infty+a = \infty + \infty$.
\end{compactenum}

If $d$ does not assume the value $\infty$, then we drop the adjective ``extended''. Finally, we call $d$ an ({\em extended}) {\em metric} if property (a) can be strengthened to the following one:
\begin{compactenum}
\item[(a$^\prime$)] for every $x,y\in X$, $d(x,y)=0$ if and only if $x=y$.
\end{compactenum}
If $d$ is a metric, then the pair $(X,d)$ (and $X$ if we do not need to specify the metric) is said to be a  {\em metric space}.

Suppose that $Y$ is a subset of $X$ and $d$ is an extended pseudo-metric. Then $d$ induces an extended pseudo-metric $d|_{Y\times Y}$ on $Y$, which we denote by $d|_Y$ for the sake of simplicity.

For a subset $A$ of a metric space $(X,d)$ we define its {\em diameter} as the value $\diam(A)=\sup_{x,y\in A}d(x,y)$, and we say that $A$ is {\em bounded} if $\diam(A)<\infty$.

Given two subsets $A,B$ of a metric space $(X,d)$, we write
$$\dist(A,B)=\inf\{d(a,b)\mid a\in A, b\in B\}.$$
If $A$ is a singleton $\{a\}$, we simply write $\dist(a,B)$ instead of $\dist(\{a\},B)$.

Next, we introduce an extended pseudo-metric on the power set $\mathcal P(\R^d)$; this pseudo-metric is the cornerstone of this paper. We denote by $\Iso(\R^d)$ the group of all isometries of $\R^d$. 
\begin{definition}
\label{DefDB}
For every $X,Y\subseteq\R^d$, we define
\begin{gather*}
d_B(X,Y)=\inf_{f\colon X\to Y\text{ bijection}} \ \sup_{x\in X}\lvert\lvert x-f(x)\lvert\lvert,\quad\text{and}\\
\EB(X,Y)=\inf_{\psi\in\Iso(\R^d)}d_B(X,\psi(Y)).
\end{gather*}
In both cases the infimum is assumed to be infinite if no bijection exists.

We call $d_B$ the {\em bottleneck distance} and $\EB$ the {\em Euclidean bottleneck distance} (\cite{EdeHeiKurSmiWin}).
\end{definition}

In the case of periodic point sets $X$ and $Y$, the minimizing bijection and isometry in Definition \ref{DefDB} exist by Propositions \ref{PropDbMIN} for $d_B$ and  \ref{PropDebMIN} for $\EB$.
We can alternatively characterize the Euclidean bottleneck distance using the following folklore result (for example, see \cite[Ch. IV, \S38]{Blu}).
\begin{theorem}\label{thm:isometry_extension}
Let $X,Y\subseteq\R^d$. If there is an isometry $\psi\colon X\to Y$, then there exists $\widetilde\psi\in\Iso(\R^d)$ such that $\widetilde\psi|_X=\psi$.
\end{theorem}
Using Theorem \ref{thm:isometry_extension}, for every $X,Y\subseteq\R^d$,
\begin{equation}\label{eq:EB}\EB(X,Y)=\inf\{d_B(i_X(X),i_Y(Y))\mid i_X\colon X\to\R^d\text{ and }i_Y\colon Y\to\R^d\text{ isometric embeddings}\}.\end{equation}

\begin{remark}
A {\em correspondence $\mathcal R$} between two sets $X$ and $Y$ is a relation $\mathcal R\subseteq X\times Y$ such that every $x\in X$ is in relation with at least one element $y\in Y$ and vice versa. Then, for every metric space $(Z,d)$, the {\em Hausdorff distance} is defined as follows: for every $X,Y\subseteq Z$,
$$d_H(X,Y)=\inf_{\mathcal R\subseteq X\times Y\text{ correspondence}}\sup_{(x,y)\in\mathcal R}d(x,y).$$

The {\em Gromov-Hausdorff distance} $\GH$ is, with an abuse of notation, an extended pseudo-metric on the class of all metric spaces. If $X$ and $Y$ are two metric spaces, then
$$\GH(X,Y)=\inf_{\text{$Z$ metric space}}\inf\{d_H(i_X(X),i_Y(Y))\mid i_X\colon X\to Z\text{ and }i_Y\colon Y\to Z\text{ isometric embeddings}\}.$$

This distance notion was introduced by Edwards (\cite{Edw}), and then rediscovered and generalised by Gromov (\cite{Gro_GH_fr}). A great impulse to study its properties recently came from applications in shape recognition and comparison (\cite{MemSap1,MemSap2,Mem07}).

If $X$ and $Y$ are two subsets of $\R^d$, the {\em Euclidean Hausdorff distance $\EH$} is another distance notion that is used to approximate the Gromov-Hausdorff distance between them considered as metric spaces. It is defined similarly to the characterisation of the Euclidean bottleneck distance given in \eqref{eq:EB}, namely,
$$\EH(X,Y)=\inf\{d_H(i_X(X),i_Y(Y))\mid i_X\colon X\to\R^d\text{ and }i_Y\colon Y\to\R^d\text{ isometric embeddings}\}.$$
Clearly, $\EH(X,Y)\geq\GH(X,Y)$. Moreover, the inequality can be strict (for example, see \cite{Mem}). Relations between $\EH$ and $\GH$ are studied in \cite{Mem} and \cite{MajVitWen}. See also \cite{Ant1,Ant2} for recent results on the Euclidean Hausdorff distance and the Gromov-Hausdorff distance on $\mathcal P(\R^d)$.

On the set $\mathcal P(\R^d)$, if we similarly define a {\em Gromov bottleneck distance} $\GB$, we have the following inequalities:
$$\xymatrix{\GH\ar@{}[r]|*[@]{\leq} \ar@{}[d]|*[@]{\leq} & \EH \ar@{}[r]|*[@]{\leq} \ar@{}[d]|*[@]{\leq} & d_H \ar@{}[d]|*[@]{\leq}\\
\GB \ar@{}[r]|*[@]{\leq} & \EB \ar@{}[r]|*[@]{\leq} & d_B.
}$$
Moreover, the inequality  $\GB\leq\EB$ can be strict as well. Consider, for example, the two triples of points given by the vertices of two concentric equilateral triangles of edges $2$ and $1$, respectively. More explicitly, take the two subsets 
$$X=\{(1,-\sqrt{3}/3),(-1,-\sqrt{3}/3),(0,2\sqrt{3}/3)\}\textrm{ and }Y=\{(1/2,-\sqrt{3}/6),(-1/2,-\sqrt{3}/6),(0,\sqrt{3}/3)\}$$ of $\R^2.$ 
Then $\EB(X,Y)=\sqrt{3}/3$. However, $\GB(X,Y)=1/2$ which can be realized by embedding $X$ and $Y$ in a metric tree as shown in Figure \ref{fig:GB}.

\begin{figure}
    \begin{subfigure}[b]{0.5\textwidth}
        \centering
        \begin{tikzpicture}[scale=1.5]
        \draw[dashed,red] (1,-0.577)--(-1,-0.577)--(0,1.154)--(1,-0.577);
            \draw[blue,dashed] (0.5,-0.2885)--(-0.5,-0.2885)--(0,0.577)--(0.5,-0.2885);
        \draw[-] (1,-0.577)--(0.5,-0.2885) (-1,-0.577)--(-0.5,-0.2885) (0,1.154)--(0,0.577) node [pos=0.5,right]{$\sqrt{3}/3$};
            \fill [red] (1,-0.577) circle (1.33pt) (-1,-0.577) circle (1.33pt) (0,1.154) circle (1.33pt);
            \fill [blue] (0.5,-0.2885) circle (1.33pt) (-0.5,-0.2885) circle (1.33pt) (0,0.577) circle (1.33pt);
            \draw[-,red] (1,-0.577)--(-1,-0.577)--(0,1.154)--(1,-0.577);
            \draw[-,blue] (0.5,-0.2885)--(-0.5,-0.2885)--(0,0.577)--(0.5,-0.2885);
        \end{tikzpicture}\label{fig:triangle}\caption{}
    \end{subfigure}
    \begin{subfigure}[b]{0.5\textwidth}
        \centering
        \begin{tikzpicture}[scale=1.5]
        \draw[-] (1,-0.577)--(0.5,-0.2885)--(0,0) (-1,-0.577)--(-0.5,-0.2885)--(0,0) (0,1.154)--(0,0.577) node [pos=0.5,right]{$1/2$} (0,0.577)--(0,0);
        \fill [red] (1,-0.577) circle (1.33pt) (-1,-0.577) circle (1.33pt) (0,1.154) circle (1.33pt);
        \fill [blue] (0.5,-0.2885) circle (1.33pt) (-0.5,-0.2885) circle (1.33pt) (0,0.577) circle (1.33pt);
        \end{tikzpicture}\label{fig:tree}\caption{}
    \end{subfigure}
    \caption{In (a), we represent the two concentric triangles in red and blue and indicate the bijection between their vertices realising $d_B(X,Y)=\sqrt{3}/3$ (and $\EB(X,Y)=\sqrt{3}/3$) in black. In (b), instead, we show the configuration of $X$ and $Y$ as subsets of a metric tree realising $\GB(X,Y)=1/2$. In this second picture, all the black segments  have  length $1/2$.}
    \label{fig:GB}
\end{figure}
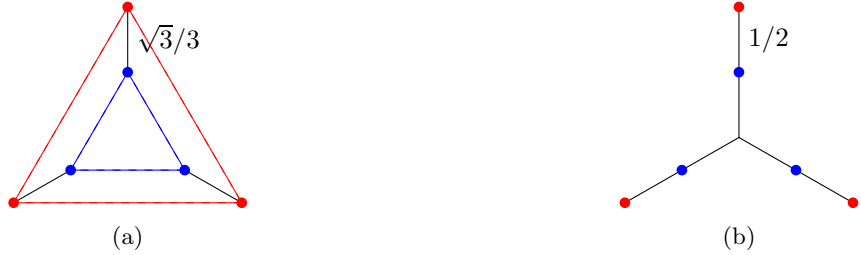
\end{remark}

Because of the existence of the minimizing isometry (Propositions \ref{PropDbMIN} for $d_B$ and  \ref{PropDebMIN} for $\EB$), it is easy to show that $\EB$ is actually an extended metric since if two periodic point sets $X$ and $Y$ satisfy $\EB(X,Y)=0$ then they are isometric. More precisely, $\EB$ is an extended metric on the space of equivalence classes of periodic point sets.

If $d$ is an extended pseudo-metric on a set $X$, and $x\in X$, following \cite{Za}, we denote by $\QQ_{(X,d)}(x)$ the ({\em large-scale}) {\em connected component of $x$}, which is the subset $\{y\in X\mid d(x,y)<\infty\}$. Clearly, $d|_{\QQ_{X}(x)}$ is a pseudo-metric.

The density can be used to characterize connected components. It is  already known that, for every pair of lattices $\Lambda,\Xi \in \Latt(\R^d)$, $d_B(\Lambda,\Xi)<\infty$ if and only if $\den(\Lambda)=\den(\Xi)$ (\cite{DunOgu}). Moreover, a similar result holds for periodic point sets and the Euclidean bottleneck distance as every periodic point is is at finite bottleneck distance from a lattice obtain by moving the points of the motif within the fundamental parallelepiped. 
\begin{theorem}\label{theo:connected_component_density}
For every $X,Y\in\PPS(\R^d)$, the following properties are equivalent:
\begin{compactenum}[(a)]
\item $\EB(X,Y)<\infty$;
\item $d_B(X,Y)<\infty$;
\item $\den(X)=\den(Y)$.
\end{compactenum}
Equivalently, for every $X\in\PPS(\R^d)$, $\QQ_{(\PPS(\R^d),\EB)}(X)=\QQ_{(\PPS(\R^d),d_B)}(X)=\PPS(\R^d,\den(X))$.
\end{theorem}
\begin{remark}
Note that the implication (c)$\to$(b) does not hold for more arbitrary Delone sets even if the sets under consideration ``look the same'' locally. Moreover, if a Delone set $X\subset \R^d$ with well-defined density (and a minor restriction on local patterns) is not at finite $d_B$-distance from some lattice, then there are uncountably many different sets that use only patterns from $X$ with pairwise infinite $d_B$-distance. We refer to \cite{fre22} for more details. 
\end{remark}

\begin{proof}
The implication (b)$\to$(a) is obvious. For the implication (c)$\to$(b) we can use the result of Duneau and Oguey \cite{DunOgu} who showed that two lattices of the same density are at finite $d_B$-distance.

In case $X$ is a PPS with a motif $P=\{x_1,\ldots,x_n\}$ of size $n$, we let $v_1,\ldots,v_d$ be a basis of the underlying lattice. For every $i\in \{1,\ldots,n\}$, we can move $x_i$ to $\frac{i}{n}v_1$. This extends to a bijection with finite $d_B$ between $X$ and the lattice with basis $\frac{1}{n}v_1,v_2,v_3,\ldots,v_d$ with the same density as $X$. After that (b) follows from (c) since the lattices obtained from $X$ and $Y$ using the described procedure will be of equal density as well.

To establish the implication (a)$\to$(c), we notice that isometries do not change density, so it is enough to show that (b) implies (c). 

Let $C_N$ be the cube of side $N$ centered at the origin and let $z:=d_B(X,Y)$. 

Then 
$$|X\cap C_N|=\den (X)\cdot N^d+O(N^{d-1}) \qquad \text{and} \qquad |Y\cap C_N|=\den (Y)\cdot N^d+O(N^{d-1}).$$
Indeed, for periodic point set $X$, we can split the space $\R^d$ into copies of the fundamental parallelepiped. Let $s$ be the diameter of the fundamental parallelepiped. Then $X\cap C_N$ contains all parallelepipeds that intersect $C_{N-s}$; this gives a lower bound for the cardinality of the intersection as 
$$|X\cap C_N|\geq |P|\cdot \frac{\Vol(C_{N-s})}{\det \Lambda}=\den (X)\cdot N^d+O(N^{d-1})$$
where $P$ is the motif, and $\Lambda$ is the underlying lattice for $X$.
Similar arguments for $C_{N+s}$ and the parallelogram it contains give an upper bound with the same leading term.

Employing similar arguments, 
$$|X\cap C_N|\geq |Y\cap C_{N-z}|$$ because every point from the latter intersection stays in the former intersection after the bijection between $Y$ and $Z$. Dividing by $N^d$ and taking the limit we get $\den(X)\geq \den(Y)$. The opposite inequality can be proved in the same way.
\end{proof}

After that, it is easy to see that $\EB$ and $\GB$ are actually extended pseudo-metric.

The situation is completely different for the Euclidean Hausdorff distance. In fact, for every $X,Y\in\PPS(\R^d)$, $\EH(X,Y)$ is bounded by the maximum between the covering radii of $X$ and $Y$ and so it is finite. Hence, the ``large-scale geometries'' induced by $\EB$ and $\EH$ are very different. We will formalize the notion of large-scale geometry in Section \ref{sec:coarse}. The opposite situation occurs for the small-scale geometry of $\PPS_r(\R^d)$.
\begin{proposition}
On $\PPS_r(\R^d)$, $\EB$ and $\EH$ induce the same topology. More precisely, the identity map $$id\colon(\PPS_r(\R^d),\EH)\to(\PPS_r(\R^d),\EB)$$ is locally an isometry (for every $s<r/2$, and $X,Y\in\PPS_r(\R^d)$, $\EH(X,Y)=s$ if and only if $\EB(X,Y)=s$). 
Similarly, $\GB$ and $\GH$ induce the same topology.
\end{proposition}
\begin{proof}
Let $X,Y\in\PPS(\R^d)$ with $\EH(X,Y)=s<r/2$.  We already know that $\EB(X,Y)\geq s$. For every $\varepsilon>0$, there is $\psi_\varepsilon\in\Iso(\R^d)$ such that $d_H(X,\psi_\varepsilon(Y))<s+\varepsilon$ and a correspondence $\mathcal R_\varepsilon$ witnessing the inequality. Because of the constraint on the packing radius of $X$ and $Y$, $\mathcal R_\varepsilon$ is forced to be one-to-one if $s+\varepsilon< r/2$, and thus $d_B(X,\psi_r(Y))<s+\varepsilon$ which implies $\EB(X,Y)<s+\varepsilon$. The proof of the  second statement is analogous.
\end{proof}
Actually, the previous result can be refined. The same reasoning can be applied to the families of all those subsets of $\R^d$ with packing radius uniformly bounded from below to conclude that the uniformities, and not just the topologies, induced by $\EB$ and $\EH$ (and $\GB$ and $\GH$) are the same (for the definition, we refer to the classic monograph \cite{Isb}).
\begin{corollary}
$\EH$ is a metric on the family of isometry classes of periodic point sets in $\PPS_r(\R^d,\kappa)$.
\end{corollary}

\subsection{Main results}\label{sec:results}

Let us now enlist the main results of this paper. Theorems \ref{theo:Latt_r_bounded}--\ref{theo:bounded_cell} provide conditions ensuring that particular families of lattices or periodic point sets are bounded with respect to the Euclidean bottleneck distance. Their proofs can be found in Section \ref{sect:BMS}. 

\begin{theorem}\label{theo:Latt_r_bounded}
$(\Latt_r(\R^d,\kappa),\EB)$ is bounded.
\end{theorem}

\begin{theorem}\label{theo:bounded_motif}
For every dimension $d$, every density $\kappa$, every relevant $R>0$, and every $N\in\N$, there exists a positive constant $K(d,\kappa,R, N)$ such that if $X_1,X_2\in\PPS^R(\R^d,\kappa)$ have representatives $X_i=\Lambda_i+P_i$ with $\lvert P_i\rvert\leq N$ for $i=1,2$, then $\EB(X_1,X_2)\leq K(d,\kappa,R,N)$.
\end{theorem}

\begin{theorem}\label{theo:bounded_cell}
For every dimension $d$, every density $\kappa$, and every $S\in\R_{>0}$, there exists a positive constant $K(d,\kappa, S)$ such that if $X_1,X_2\in\PPS(\R^d,\kappa)$ have representatives $X_i=\Lambda_i+P_i$ with $\diam(U(\Lambda_i))\leq S$ for $i=1,2$, then $\EB(X_1,X_2)\leq K(d,\kappa,S)$.
\end{theorem}

\begin{remark}   
Note that bounding packing radius, covering radius, or the diameter of the fundamental region is essential for the theorems above.

If we drop these requirements even for $d=2$, then the integer lattice $\Z^2$ and the ``stretched'' integer lattice $a\Z\times \frac{1}{a}\Z$ have the same density but can be very far in $d_{EB}$-distance apart if the positive parameter $a$ is very large or very small.
\end{remark}

We then investigate the following problem: is it true that the space of periodic point sets equipped with the Euclidean bottleneck distance is separable? It turns out that even the space of lattices is not. More precisely, we prove the following, stronger result (see Section \ref{sec:bi-Lipschitz} for its proof).

\begin{theorem}\label{theo:uncountable_family} 
Given dimension $d$ and density $\kappa$, there exist $r,R$ and an uncountable family of lattices $\{\Lambda_\alpha\}_{\alpha\in I}\subseteq\Latt_r^R(\R^d,\kappa)$ and $S\in\R_{>0}$ such that $\EB(\Lambda_\alpha,\Lambda_\beta)\geq S$ for every $\alpha,\beta\in I$.
\end{theorem}
In particular, that construction implies that $\Latt_r^R(\R^d,\kappa)$ is not separable, and the same holds for every other family of periodic point sets containing it. Let us recall that a metric space is {\em separable} if it has a countable dense subset. {\color{red} }

Before stating the embeddability results, let us first recall that a map $f\colon(X,d_X)\to(Y,d_Y)$ between metric spaces is a
\begin{compactenum}[(a)]
\item {\em bi-Lipschitz embedding} if there exists $L>0$ such that, for every $x,x^\prime\in X$,
$$L^{-1}\cdot d_X(x,x^\prime)\leq d_Y(f(x),f(x^\prime))\leq L\cdot d_X(x,x^\prime);$$
\item {\em coarse embedding} if there exist two monotonous maps $\rho_-,\rho_+\colon\R_{\geq 0}\to\R_{\geq 0}$ such that $\lim_{t\to\infty}\rho_-(t)=\infty$ and, for every $x,x^\prime\in X$,
$$\rho_-(d_X(x,x^\prime))\leq d_Y(f(x),f(x^\prime))\leq \rho_+(d_X(x,x^\prime)).$$
\end{compactenum}
If we need to specify the constant $L>0$ witnessing that $f$ is a bi-Lipschitz embedding, we say that $f$ is a {\em $L$-bi-Lipschitz embedding}. Similarly, in order to emphasize the two functions $\rho_-$ and $\rho_+$, called {\em control functions},  we call $f$ a {\em $(\rho_-,\rho_+)$-coarse embedding}.

Let us enlist a few immediate facts concerning the just-defined embeddings.
\begin{remark}\label{rem:embeddings}
\begin{compactenum}[(a)]
\item A bi-Lipschitz embedding is, in particular, a homeomorphism onto its image. Thus, as a consequence, we have that a non-separable metric space cannot be bi-Lipschitz embedded into any separable spaces.
\item Clearly, a  bi-Lipschitz embedding is, in particular, a coarse embedding.
\item If a metric space is bounded, then the map collapsing it to a one-point space is a  coarse embedding.
\end{compactenum}
\end{remark}

Let us recall that a Banach space $(A,\lvert\lvert\cdot\rvert\rvert)$ is {\em uniformly convex} if, for every $0<\varepsilon\leq 2$, there exists $\delta>0$ so that, for any two vectors $x,y\in A$ with $\lvert\lvert x\rvert\rvert=\lvert\lvert y\rvert\rvert=1$, the condition $\lvert\lvert x- y\rvert\rvert\geq\varepsilon$ implies that $\lvert\lvert(x+y)/2\rvert\rvert\leq 1-\delta$. Hilbert spaces are, in particular, uniformly convex Banach spaces.

\begin{theorem}\label{theo:no_coarse_embedding}
There is no coarse embedding of $\PPS_r(\R^d,\kappa)$ into a uniformly convex Banach space. Similarly, there is no coarse embedding of $\PPS^R(\R^d,\kappa)$ into a uniformly convex Banach space.
\end{theorem}
The conclusions of Theorem \ref{theo:no_coarse_embedding} hold if the spaces are endowed with either the bottleneck or the Euclidean bottleneck distances.

\begin{remark}
We do not emphasize the bounds on the packing and covering radii in the theorem above. Provided the density $\kappa$ is fixed, our proofs in Section \ref{sec:coarse} show that the statement holds for sufficiently small packing radius $r$ and sufficiently large covering radius $R$. Moreover, $r$ can be chosen $\epsilon$-close to the best packing radius and $R$ can be chosen $\epsilon$-close to the best covering radius for $d$-dimensional lattices of density $\kappa$.

We note that we cannot use all radii as, for example, in dimensions 8 and 24, the lattice $E_8$ and the Leech lattice $\Lambda_{24}$ are the only periodic point sets with maximal densities in the corresponding dimensions due to Viazovska \cite{Via17} and Cohn, Kumar, Miller, Radchenko, and Viazovska \cite{CKMRV17}.
\end{remark}

We summarize some consequences of the above theorems in Table \ref{table:embeddability_results}. In particular, we enlist those regarding the existence of particular embeddings of subspaces of $\PPS(\R^d,\kappa)$ into Hilbert spaces. Theorem \ref{theo:Latt_r_bounded} implies that the map collapsing $\Latt_r(\R^d,\kappa)$ to a point is a coarse embedding. The same statement holds for subspaces of $\PPS(\R^d,\kappa)$ with bounded either unit cells  or cardinality of the motifs according to Theorems \ref{theo:bounded_motif} and \ref{theo:bounded_cell}. Since $\Latt_r^R(\R^d,\kappa)$ and the spaces of periodic point sets containing it as a subspace are not separable (Theorem \ref{theo:uncountable_family}), they cannot be bi-Lipschitz embedded into any separable space (in fact, bi-Lipschitz embeddings are topological embeddings), and so in particular, into any separable Hilbert space. Finally, Theorem \ref{theo:no_coarse_embedding} states that every space of periodic point sets containing either $\PPS_r(\R^d,\kappa)$  or $\PPS^R(\R^d,\kappa)$ cannot be coarsely embedded (and, thus, in particular, they cannot be bi-Lipschitz embedded) into any Hilbert space.

	\begin{table}[h]
	\centering
		\begin{tabular}{ c | c c || c | c c}
			lattices & bi-Lipschitz & coarse & periodic point sets & bi-Lipschitz & coarse\\
			\hline
			$\Latt(\R^d)$ & no ($\omega$)  &  & $\PPS(\R^d)$ & no & no  \\ 
			$\Latt_r(\R^d)$ & no ($\omega$) & yes & $\PPS_r(\R^d)$ & no & no \\  
			$\Latt^R(\R^d)$ & no ($\omega$) & yes & $\PPS^R(\R^d)$ & no & no  \\
			$\Latt_r^R(\R^d)$ & no ($\omega$) & yes & $\PPS_r^R(\R^d)$ & no ($\omega$) & \\
			\begin{tabular}{@{}c@{}}$\Latt(\R^d)$ with\\bounded diameter\\unit cells\end{tabular}  & no ($\omega$) &  yes & \begin{tabular}{@{}c@{}}$\PPS(\R^d)$ with\\bounded diameter\\unit cells\end{tabular} & no ($\omega$) & yes \\
			& & & \begin{tabular}{@{}c@{}}$\PPS(\R^d)$ with\\bounded cardinality\\motifs\end{tabular} & no ($\omega$) & yes
		\end{tabular}
		\caption{In this table we collect answers (formally stated in Theorems \ref{theo:bounded_cell}--\ref{theo:uncountable_family} and \ref{theo:no_coarse_embedding} and Remark \ref{rem:embeddings}) to the following question: can a subspace $\mathcal X$ of $\PPS(\R^d,\kappa)$ equipped with either $d_B$ or $\EB$ be bi-Lipschitz or coarsely embedded into a Hilbert space? A ``no'' with an $\omega$ between brackets stands for the non-embeddability if the Hilbert space has a countable base.}\label{table:embeddability_results}
	\end{table}
	
The following question remains open.	
\begin{question}
Can $\Latt(\R^d,\kappa)$ and $\PPS_r^R(\R^d,\kappa)$ be coarsely embedded into a Hilbert space? 
\end{question}

\section{Boundedness}
\label{sect:BMS}

The main purpose of this section is to prove Theorem \ref{theo:Latt_r_bounded}. We start by proving some preliminary lemmas for lattices with arbitrary densities. We assume that $d\geq 2$.

One of our main tools is splitting lattice $\Lambda$ into {\it cosets} with respect to a sublattice $\Gamma$ of smaller dimension. We can treat $\Lambda$ and $\Gamma$ as groups under addition of vectors, and in that setting $\Gamma$ is a subgroup of $\Lambda$. Then the cosets are defined in the same way as cosets with respect to a subgroup.

\begin{lemma}\label{lem:projection}
    Let $\Lambda$ be a $d$-dimensional lattice with density $\rho$ and packing radius $r$. Suppose $v\in \Lambda$ is its shortest vector and $\pi$ is a hyperplane such that $v\notin\pi$. Let us denote by $\alpha\in(0,\pi/2]$ the angle between $v$ and $\pi$.

    Then the projection of $\Lambda$ on $\pi$ along $ v$ is a $(d-1)$-dimensional lattice with density $2\rho r\sin \alpha $ and packing radius at least $\frac{r\sqrt 3}{2}$.
\end{lemma}
\begin{proof}
    If we choose a basis of $\Lambda$ that includes $ v$, then the projection $\Lambda'$ is a lattice generated by the projected vectors of the basis (except $ v$ which is projected into the zero vector).

    The statement about density follows from the same basis. The volume of the parallelepiped spanned by this basis of $\Lambda$ is $\sfrac{1}{\rho}$. Projection of this parallelepiped on a hyperplane orthogonal to $ v$ has $(d-1)$-volume $\frac{1}{\rho | v|}$ and therefore its projection on $\pi$ along $ v$ has $(d-1)$-volume $\frac{1}{\rho | v|\sin \alpha }$. This gives the claimed density as $| v|=2r$.

    In order to establish the bound for packing radius, we split $\Lambda$ into cosets parallel to $ v$. Each such coset belongs to a line in $\mathbb{R}^d$ parallel to $ v$ and results in one point of the projected lattice $\Lambda'$.

    The distance between two such lines is at least $r\sqrt 3$ as there are two points of $\Lambda$ with distance less than $2r$ otherwise. Indeed, if in Figure \ref{fig:projection} the base of the triangle is $2r$, and two sides opposite to acute angles next to the base are at least $2r$, then the height is at least $r\sqrt 3$. The distance between projections of these two lines along $ v$ can not decrease.
    \begin{figure}
        \centering
        \includegraphics[width=0.4\textwidth]{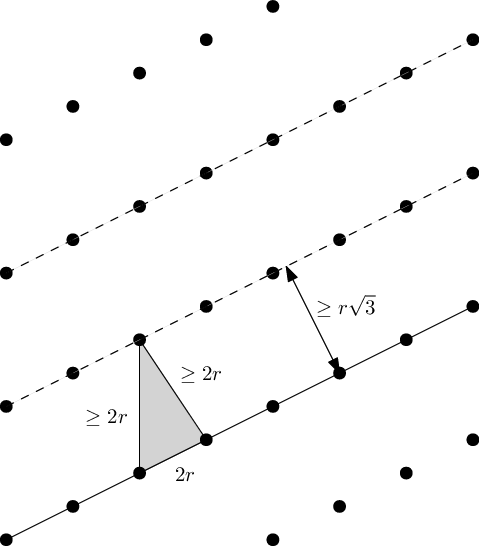}
        \caption{Cosets of $\Lambda$ parallel to the shortest vectors $ v$.}
        \label{fig:projection}
    \end{figure}
\end{proof}

Our next goal is to show that the $d_B$-distance between two $d$-dimensional lattices of density $1$ with certain common bound on the packing radii does not exceed a constant that depends on the dimension $d$ and that radius.

The approach is similar to the one used by Duneau and Oguey \cite{DunOgu} to prove that such two lattices are at finite bounded distance but we track the distances for bijections used in the process.

It is enough to prove a similar statement for one lattice and the $d_B$-distance to the integer lattice $\Z^d$.

\begin{lemma}\label{lem:bounded_packing}
    Let $\Lambda$ be a $d$-dimensional lattice of packing radius at least $r$ and density $\kappa=\rho^d$. Then there exists a constant $c=c(d,r,\rho)$ such that 
    $$d_B(\Lambda,\frac{1}{\rho}\Z^d)\leq c.$$
    Additionally, we can take the constant $c$ to be non-increasing in $r$ and $\rho$.
\end{lemma}
\begin{proof}
Note that the statement is obvious for $d=1$ as there is only one one-dimensional lattice of each positive density so we assume $d\geq 2$ and prove it by induction for such dimensions.

We also notice right away that for every fixed dimension, it is enough to prove the existence statement for $\rho=1$ and the $d_B$ distance to the integer lattice $\Z^d$. Indeed, the existence of the constant $c$ for arbitrary lattice $\Lambda$ with density $\kappa=\rho^d$ follows from the equality
$$d_B\bigg(\Lambda,\frac{1}{{\rho}}\Z^d\bigg)=\frac{1}{\rho}d_B({\rho} \Lambda,\Z^d)$$
where the distance in the right hand side is bounded by $c(d,\rho r, 1)$. The monotonicity in $\rho$ also follows from the same equality because of the factor $\frac1{\rho}$. And the monotonicity in $r$ follows from the property that decreasing $r$ increases the family of lattices we consider, and therefore doesn't decrease $c$ as function of $r$.

We are ready to start our induction argument. The basis of induction is $d=2$. Let $\Lambda$ be a two-dimensional lattice with density 1 and with packing radius at least $r$. Let $ v$ be a shortest nonzero vector of $\Lambda$. We note that there is the maximal packing radius $r_2$ among all two-dimensional lattices of density $1$. This radius is achieved by the hexagonal lattice $\mathsf D_2$
, but we don't need the exact value here. Thus,  $2r\leq | v|\leq 2r_2$.

The process that we use to construct a bijection between $\Lambda$ and $\Z^2$ can be described as follows. At each step, we pick a primitive vector $ x$
in the current lattice, i.e. a vector of a lattice basis, and shift all lattice points along $ x$ to get a new lattice. We will call this process a {\it coset shift along} $ x$. Notice that the new lattice is at $d_B$-distance at most $| x|$ from the  old lattice. The process is illustrated in Figure \ref{fig:coset}.

\begin{figure}
    \centering
    \includegraphics[width=0.45\textwidth]{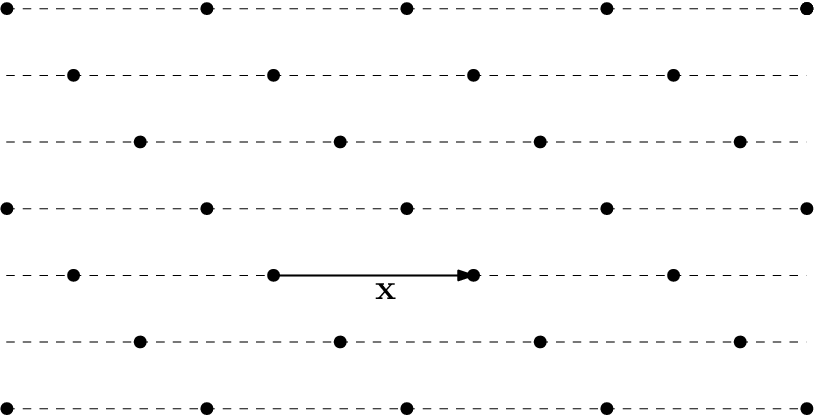}
    \hskip 0.05\textwidth
    \includegraphics[width=0.45\textwidth]{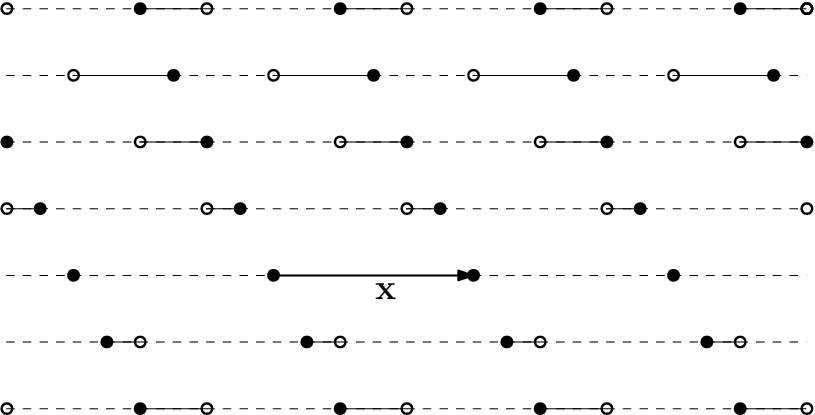}
    \caption{Coset shift. The left panel shows the initial lattice and its vector $ x$ with the corresponding cosets. The right panel shows how each coset shifts from white dots to black dots to form a new lattice.}
    \label{fig:coset}
\end{figure}

Let $ e_1, e_2$ be the standard orthonormal basis of $\Z^2$ (and $\R^2$). We assume that the angle $\alpha$ between $ e_2$ and $ v$ is in $[\sfrac{\pi}{4},\sfrac{\pi}{2}]$ as we can switch the roles of $ e_1$ and $ e_2$ or/and their signs otherwise.

Let $ w= v- e_2$. 
Note that this vector is not parallel to $ v$. We shift cosets in $\Lambda$ along $ v$ to get a new lattice $\Lambda_1$ that contains a vector parallel to $ w$ that connects two lines connecting two consecutive cosets of $\Lambda$ parallel to $ v$, see Figure \ref{fig:2d-shifts}.

\begin{figure}
    \centering
    \noindent
    \frame{\includegraphics[width=0.4\textwidth]{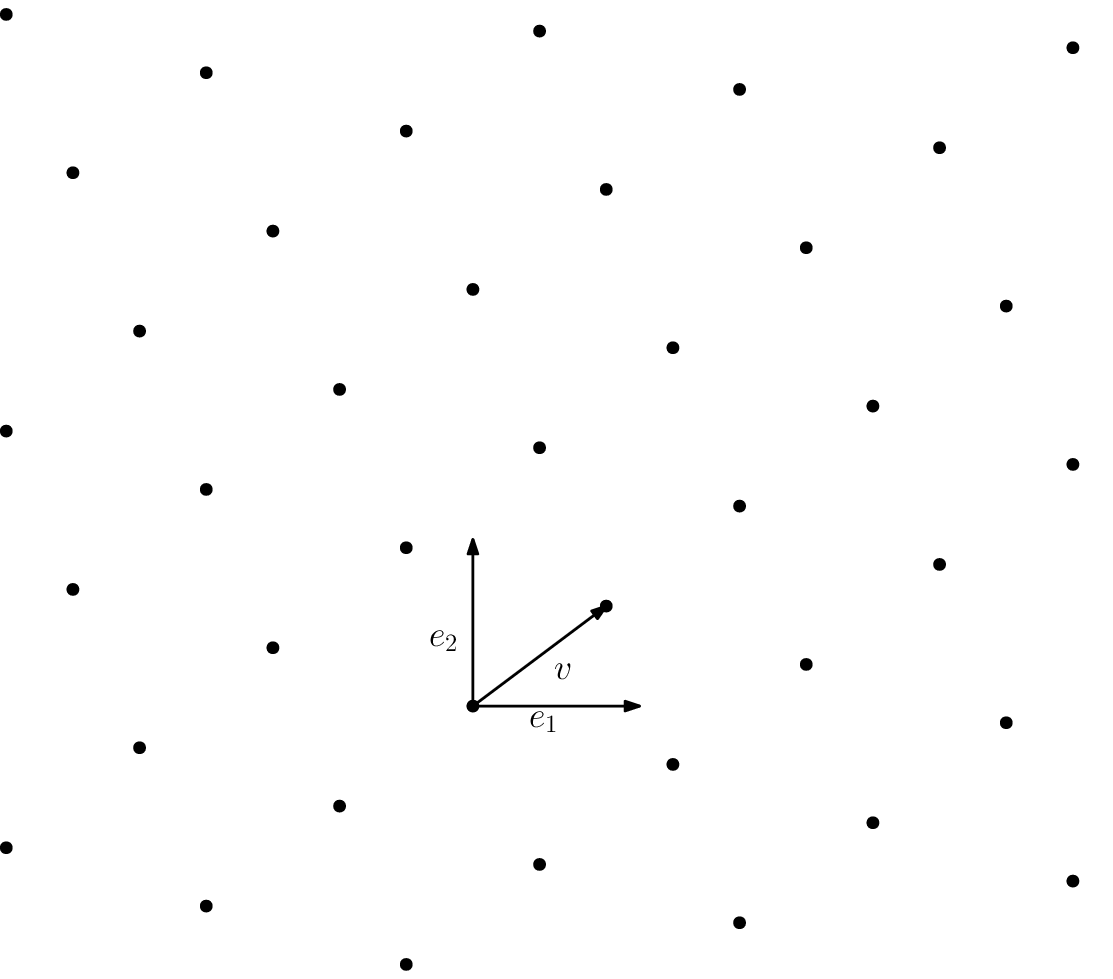}}
    \hskip 0.05\textwidth
    \frame{\includegraphics[width=0.4\textwidth]{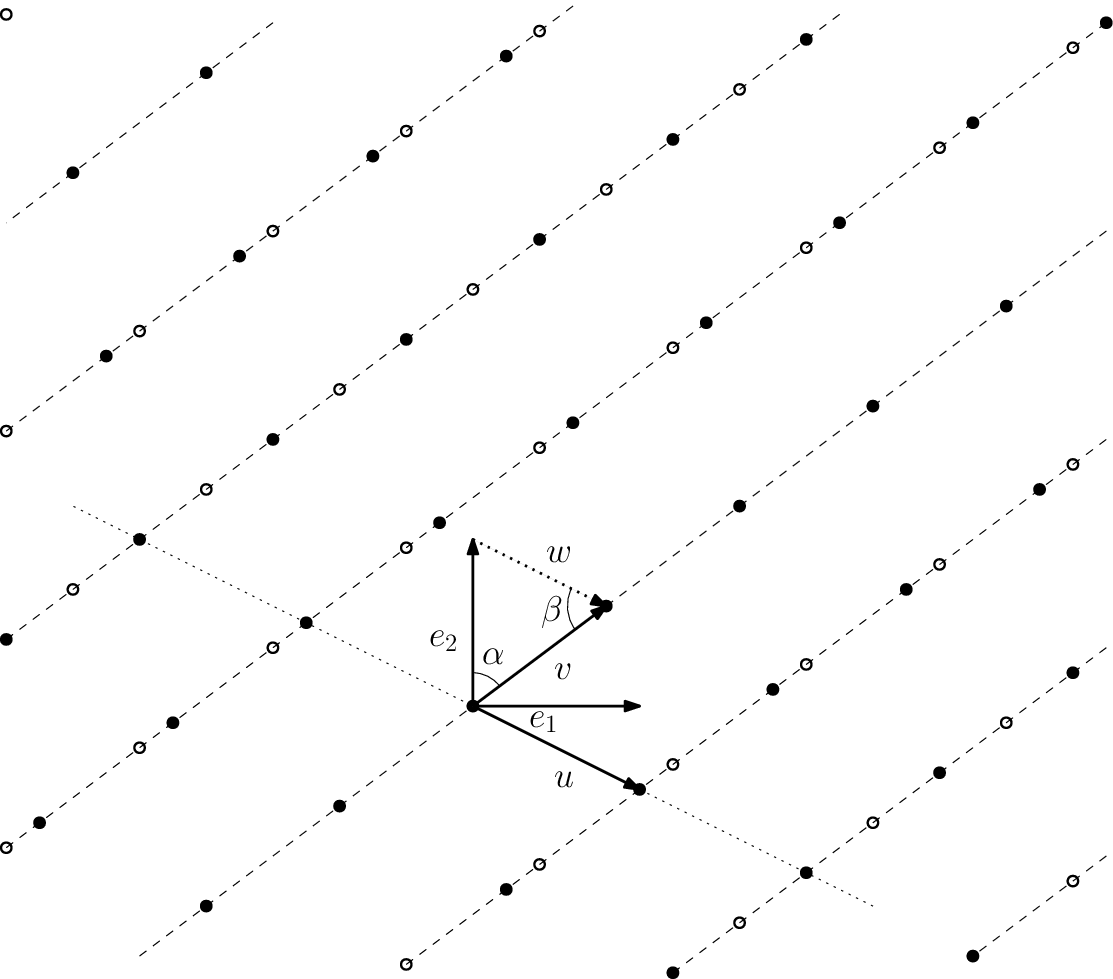}} \\ \vskip1em
    \frame{\includegraphics[width=0.4\textwidth]{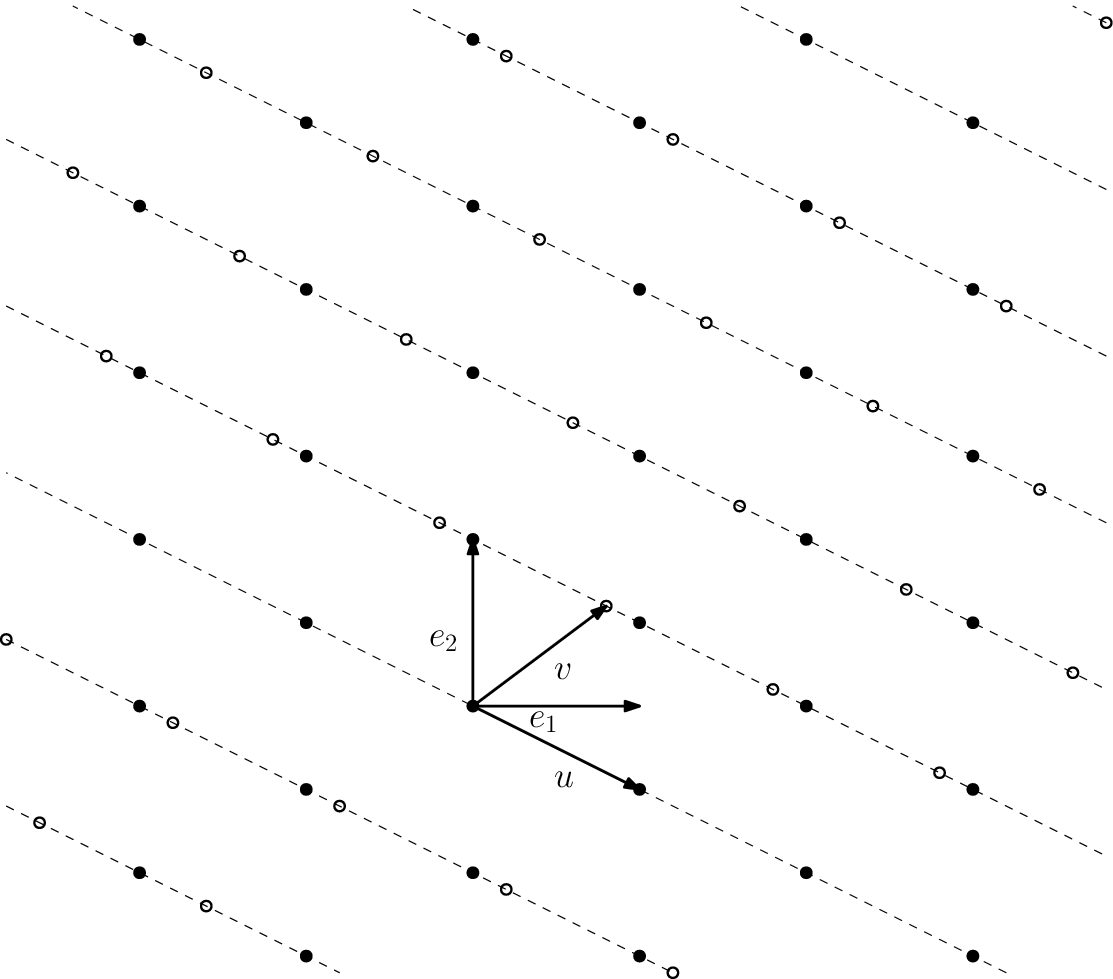}}\hskip 0.05\textwidth
    \frame{\includegraphics[width=0.4\textwidth]{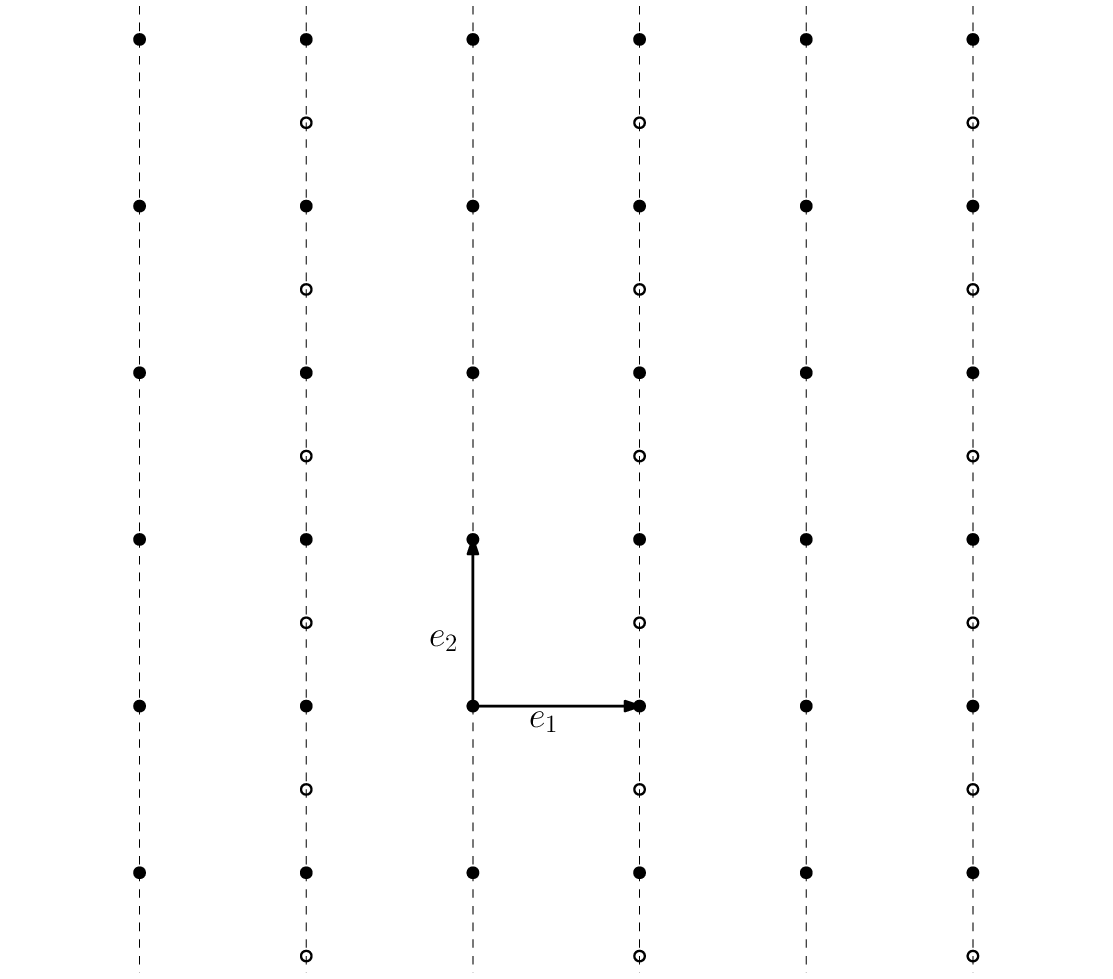}}
    \caption{The sequence of coset shifts for the two-dimensional proof of Lemma \ref{lem:bounded_packing}. The order of the pictures is from left to right and them from top to bottom. In each picture we shift white dots to black dots along dashed lines. The final lattice of solid points is the integer lattice.}
    \label{fig:2d-shifts}
\end{figure}

Let $ u$ be the (shortest) vector of $\Lambda_1$ parallel to $ w$ and let $\beta$ be the angle between $ v$ and $ u$. From the law of sines for the triangle with sides $ e_2$ and $ v$,
$$\sin \beta = \frac{\sin \alpha}{| w|}\geq \frac{\sfrac{1}{\sqrt 2}}{1+| v|}\geq \frac{\sqrt 2}{2+4r_2}.$$

Since $ u$ and $ v$ give a basis of $\Lambda_1$ and the lattice has density $1$, $$| u|=\frac{1}{| v|\sin\beta}\leq \frac{1+2r_2}{r\sqrt 2}.$$

Shifting the cosets of $\Lambda_1$ along $ u$ we get a lattice $\Lambda_2$ that has $ e_2$ as a primitive vector, see Figure \ref{fig:2d-shifts}.

Shifting cosets of $\Lambda_2$ along $ e_2$ we get $\Z^2$, see Figure \ref{fig:2d-shifts}.

Overall, we used three coset shifts using vectors $ v$, $ u$, and $ e_2$. This implies that
$$d_B(\Lambda,\Z^2)\leq | u|+| v|+| e_2|\leq 2r_2+\frac{1+2r_2}{r\sqrt 2}+1$$
which proves the existence of the constant $c(2,r,1)$.

The induction step for $d\geq 3$ uses a similar coset shift technique for shifts within (or along) sublattices of smaller dimension with the bound on the $d_B$-distance ensured by the induction hypothesis.

We notice that there is a supremum $r_d$ of all packing radii of $d$-dimensional lattices of density $1$.

Let $\Lambda$ be a $d$-dimensional lattice of density $1$ and let $ v$ be its shortest vector. Its magnitude satisfies $2r\leq | v|\leq 2r_d$. We look at the coordinate hyperplanes in $\R^d$, without loss of generality we can assume that the angle $\alpha$ between $ v$ and the hyperplane $x_1=0$ satisfies $\sin\alpha\geq \sfrac{1}{\sqrt d}$.

Let $\Lambda'$ be the projection of $\Lambda$ on $x_1=0$ along $ v$. According to Lemma \ref{lem:projection}, $\Lambda'$ is a $(d-1)$-dimensional lattice with density $2r\sin \alpha$ and packing radius at least $\frac{r\sqrt 3}{2}$.

We first apply the coset shift of $\Lambda$ along $ v$ that results in the lattice $\Lambda_1$ spanned by $ v$ and $\Lambda'$.

Our next step is to transform $\Lambda'$ into the lattice $\Lambda''$ spanned by $\frac{1}{2r\sin \alpha}  e_2, e_3,\ldots, e_d$. By induction hypothesis,
$$d_B(\Lambda',\Lambda'')\leq 2c\Big(d-1, \min\Big(\frac{r\sqrt 3}{2},\frac{1}{2},\frac{1}{4r\sin \alpha}\Big),2r\sin \alpha\Big)\leq 2c\Big(d-1,\min\Big(\frac{r}{2},\frac{1}{2},\frac{1}{4r_d}\Big),\frac{r}{\sqrt d}\Big)$$
where the factor $2$ comes from considering $d_B$-distance to the appropriate scaling of integer lattice from both $\Lambda'$ and $\Lambda''$ and the inequality between constants $c$ comes from decreasing both density and allowed packing radius.

We can use this transformation in every coset of $\Lambda'$ in $\Lambda_1$ treating the multiples of $ v$ in these cosets as ``fixed'' origin. This transforms $\Lambda_1$ into the lattice $\Lambda_2$ spanned by $ v$ and $\frac{1}{2r\sin \alpha}  e_2, e_3,\ldots, e_d$.

Next, we shift cosets of $\Lambda_2$ along its $(d-2)$-dimensional sublattice spanned by $ e_3,\ldots, e_d$. 
That is, we shift each coset as a rigid structure but the shifts of separate cosets are different to get a lattice as the outcome. This gives us the lattice $\Lambda_3$ generated by these $d-2$ vectors, $\frac{1}{2r\sin \alpha}  e_2$, and the projection of $ v$ on the two-dimensional plane spanned by $ e_1, e_2$. Since each such coset is geometrically a $(d-2)$-dimensional integer lattice, the magnitude of each shift is at most $\sqrt d$.

We note that the $ e_1$-component of the projection of $ v$ on the plane of $ e_1$ and $ e_2$ is at least $\sfrac{2r}{\sqrt d}$ in absolute value. Thus the sublattice of $\Lambda_3$ in the two-dimensional plane of $ e_1$ and $ e_2$ has packing radius at least $\min(\frac{1}{2r\sin \alpha},\frac{r}{2\sqrt d})\geq \min(\frac{1}{2r},\frac{r}{\sqrt d})$ and density 1 since the density of all our $d$-dimensional lattices are $1$.

Applying the two-dimensional case of our lemma to this sublattice and the corresponding cosets we get the inequality

$$d_B(\Lambda,\Z^d)\leq | v|+2c\Big(d-1,\min\Big(\frac{r}{2},\frac{1}{2},\frac{1}{4r_d}\Big),\frac{r}{\sqrt d}\Big)+\sqrt{d}+c\Big(2,\min\Big(\frac{1}{2r},\frac{r}{\sqrt d}\Big),1\Big)$$
which ensures existence of $c(d,r,1)$ as needed for the induction step.
\end{proof}

\begin{proof}[Proof of Theorem \ref{theo:Latt_r_bounded}.]
The theorem is an immediate corollary of Lemma \ref{lem:bounded_packing} because the lemma claims that the distance between two lattices in  $(\Latt_r(\R^d,\kappa),d_{EB})$ is bounded by some constant depending on $d$, $r$, and $\kappa$.
\end{proof}

Using Theorem \ref{theo:Latt_r_bounded}, the counterpart of that result for lattices with uniformly bounded covering radius can be deduced.
\begin{corollary}\label{cor:covering}
    $(\Latt^R(\R^d,\kappa),d_{EB})$ is bounded.
\end{corollary}
\begin{proof}
    We claim that if the covering radius of a $d$-dimensional lattice with density $\kappa$ is bounded above by $R$, then its packing radius is bounded from below by some constant $r=r(d,R,\kappa)$.

    Assuming the contrary, let $\Lambda_n, n\in \mathbb N$ be a family of lattices in $\Latt^R(\R^d,\kappa)$ with packing radii $r=\frac{1}{n}$, respectively.

    We project each lattice $\Lambda_n$ along its shortest vector on an orthogonal hyperplane. The projected lattices have covering radii at most $R$ and density going 0 as $n$ goes to infinity from Lemma \ref{lem:projection}. This is a contradiction because the ($(d-1)$-dimensional) sets with covering radii at most $R$ must have points in every cube with side $2R$ and therefore have densities bounded away from 0.

    Now it is enough to use the boundedness of $\Latt_r(\R^d,\kappa)$ in the $d_{EB}$-distance.
\end{proof}

\begin{proof}[Proof of Theorem \ref{theo:bounded_motif}.]
We sketch the proof and leave out technicalities.

Recall that we aim to prove existence of $K=K(d,\kappa,R, N)$ such that two sets $X_i\in \PPS^R(\R^d,\kappa)$, $i=1,2$, with at most $N$ points in their motifs satisfy $d_{EB}(X_1,X_2)\leq K$.

First we prove that there exists $R'=R'(d,\kappa,R,N)$ such that the covering radii of $\Lambda_i,i=1,2$, see the statement of Theorem \ref{theo:bounded_motif}), do not exceed $R'$.

We show that the covering radius of $\Lambda_i$ is at most $(N+1)R$ If the covering radius of $\Lambda_i$ is greater than $(N+1) R$, then there is a ball of radius $(N+1) R$ in $\R^d$ with the origin $0\in\R^d$ on the boundary of this ball that does not contain points of $\Lambda_i$ inside. On the other hand, this ball contains $N+1$ nonintersecting balls of radius $R$ aligned along its diameter containing the origin, and since the covering radius of $X_i$ is at most $R$, two of these balls contain points corresponding to the same point of the motif. Then the difference of the said two points lies inside the same empty ball of radius $(N+1) R$.
Hence, a contradiction and the covering radius of $\Lambda_i$ is bounded.

Since we established that the covering radius is bounded, we may transform $X_i$ for $i=1,2$ into some lattices $\Lambda_i'$  by shifting each subset generated by a single point of the motif (each such subset is at bottleneck distance at most $R'$ from $\Lambda_i$) similarly to what we did in the proof of Theorem \ref{theo:connected_component_density}, and can use Corollary \ref{cor:covering} for $\Lambda_i'$.
\end{proof}

\begin{proof}[Proof of Theorem \ref{theo:bounded_cell}.]
Again, we sketch the proof and leave out technicalities.

For $i=1,2$, we move points of the motifs of $X_i$ within the corresponding fundamental cells to get lattices. This changes each $X_i$ by at most $S$ in $d_{EB}$-distance.

The resulting two lattices have equal densities, and diameters of fundamental cells at most $S$ and therefore the packing radii at most $S$. Now we can use Corollary \ref{cor:covering} to get the claim.
\end{proof}

\section{Nonseparability}\label{sec:bi-Lipschitz}

We start by constructing an uncountable family of lattices that we use to prove Theorem \ref{theo:uncountable_family}.

For $d\geq 2$, let $\alpha\in [0,\sfrac 12]$ be a real number. We define the $d$-dimensional lattice $\Gamma_\alpha$ as 
$$\Gamma_\alpha=\mathbb{Z}(e_1, e_2+\alpha e_1, e_3, \ldots,e_{d}).$$
In other words, this lattice has a layered structure where each layer is the $(d-1)$-dimensional integer lattice spanned by vectors $e_1$ and $e_3,\ldots, e_d$ located in the plane $x_2=k$ for some $k\in \mathbb Z$. Distance between hyperplanes that contain two consecutive layers is $1$, and two consecutive layers are shifted by $e_2+\alpha e_1$ with respect to each other.

Also note that the density of $\Gamma_\alpha$ is 1 and does not depend on $\alpha$. Thus, the whole family $\{\Gamma_\alpha\}_{\alpha\in [0,\frac 12]}$ belongs to one family $\Latt_r^R(\R^d,1)$ for appropriate $r$ and $R$.

We will use the following simple property.

\begin{lemma}\label{lem:fractional}
Let $a,\lambda\in\R$. If $\lambda$ is not an integer, then every arithmetic progression $\{a+b\lambda\mid b\in \mathbb Z\}$ contains a number with fractional part in every closed segment of $\mathbb R/\mathbb Z$ of length $\frac12$.
\end{lemma}
\begin{proof}
If $\lambda$ is irrational then the corresponding fractional parts are dense in $[0,1)$. Otherwise, the fractional parts form a discrete set with distance $\frac 1q$ between consecutive points for some integer $q\geq 2$ and the lemma follows as $\frac 12\geq \frac 1q$. 
\end{proof}

\begin{lemma}\label{lem:point14}
Let $u$ be a unit vector not parallel to each $e_1,e_3,e_4,\ldots, e_{d}$. Then for every $x\in\R^d$ and for every $\alpha\in (0,\sfrac12]$, there is a point $y\in \{x+n\cdot u\mid n\in \mathbb Z\}$ such that 
$$\dist(\Gamma_\alpha,y)\geq \frac14.$$

\end{lemma}
\begin{proof}
First we assume that $d>2$ and reduce the lemma to the two-dimensional case.

We look at the $3$rd coordinate of the points in $\{x+n\cdot u:n\in \mathbb Z\}$. They form an arithmetic progression with difference $\lambda\in [0,1]$ 
where $\lambda$ is the absolute value of the $3$rd coordinate of $u$. If $\lambda$ is not an integer, then this arithmetic progression contains a number with fractional part  in $[\frac 14,\frac 34]$ due to Lemma \ref{lem:fractional}. In this case the distance from  the corresponding point to $\Gamma_\alpha$ is at least $\frac 14$ because the $3$rd coordinate of every point in $\Gamma_\alpha$ is integer. If $\lambda$ is an integer then $\lambda=0$ or $\lambda=1$, but the latter case is impossible because $u$ is not parallel to $e_3$. Thus $\lambda=0$. Similar arguments show that all coordinates of $u$ except the first two are zeros. Changing the corresponding coordinates of $x$ to zeroes we reduce the lemma to the two-dimensional case as this change does not increase the distances to $\Gamma_\alpha$ for all points in $\{x+n\cdot u\mid n\in \mathbb Z\}$.

From now on we assume that $d=2$.

Still, every point in $\Gamma_\alpha$ has integer $2$nd coordinate, so by the same arguments the second coordinate of $u$ must be an integer. If it is $0$, then $u=\pm e_1$ giving a contradiction, so $u=\pm e_2$. Without loss of generality we can assume that $x=\beta e_1$ as shifting the set $\{x+n\cdot u\mid n\in \mathbb Z\}$ along $e_2$ to the closest layer of $\Gamma_\alpha$ parallel to $e_1$ does not increase the distance we are bounding from below.

Since $\alpha\neq 0$, then by Lemma \ref{lem:fractional} we can find an integer $k$ such that distance from $\{k\alpha+n\mid n\in\mathbb Z\}$ to $\beta$ is at least $\frac 14$ as we can find $k$ such that the $\{k\alpha\}$ is in $[\beta +\frac 14,\beta +\frac 34] \mod 1$. Considering the point in $x+k\cdot u=\beta e_1+k e_2$ which is the closest to the corresponding subset $\{k(e_2+\alpha e_1)+n\cdot e_1\mid n\in \mathbb Z\}$ of $\Gamma_\alpha$, we get the desired inequality.
\end{proof}

\begin{lemma}
\label{LemDifference}
Suppose $\alpha,\beta \in [0,\sfrac 12]$ are two different numbers. Then $d_{EB}(\Gamma_\alpha,\Gamma_\beta)\geq \frac 14$.
\end{lemma}
\begin{proof}
Without loss of generality we assume that $\alpha\neq 0$. It is enough to show that for every Euclidean isometry $\phi$, there exists a point $x\in \Gamma_\alpha$ such that $\dist(x,\phi(\Gamma_\beta))\geq \frac 14$.

Let $\phi'$ be the orthogonal transformation induced by $\phi$ on vectors. Let $v_i=\phi'(e_i)$ for $i=1,3,4,\ldots, d$. Then each $v_i$ is a unit vector and $\phi(\Gamma_\beta)$ contains the set $\{\phi(0)+n\cdot v_i\mid n\in \mathbb{Z}\}$. Therefore, $v_i$ must be parallel to some vector of $e_1,e_3,e_4,\ldots,e_d$ otherwise we can use Lemma \ref{lem:point14}. Note that every one of these $d-1$ vectors (up to sign) must be present in $v_1,v_3,v_4,\ldots,v_d$. 

Also, $v_2:=\phi'(e_2)=\pm e_2$ as this image must be orthogonal to all $v_1,v_3,v_4,\ldots,v_d$.

If $\beta=0$, then we can apply the same lemma to vector $\phi'(e_2)$ and it cannot be in $e_1,e_3,e_4,\ldots,e_d$ as it must be orthogonal to all $v_i$, $i=1,3,4,\ldots, d$. Then Lemma \ref{lem:point14} implies that the Euclidean bottleneck distance between $\Gamma_\alpha$ and $\Gamma_\beta$ is at least $\frac 14$.

If $\beta\neq 0$ then we consider two cases. If $v_1=\pm e_i$ for $i\geq 3$, then the $i$th coordinates of points in $\phi(\Gamma_\beta)$ 
form the set $\{x\pm\beta n+m\mid m,n\in \mathbb Z\}$ 
where $x$ is the $i$th coordinate of $\phi(0)$. By Lemma \ref{lem:fractional}, this set contains a number with fractional part in $[\frac14,\frac 34]$ and the lemma follows. 

In the second case  $v_1=\pm e_1$. Recall that $v_2=\phi'(e_2)=\pm e_2$. In this situation, we can assume that $d=2$ as ignoring the last $d-2$ coordinates does not increase the distance. We can also assume that $v_1=e_1$ and $v_2=e_2$ as three other situations can be treated with similar arguments, possibly changing signs of some vectors.

For every $n\in \mathbb Z$ we consider the following function $f(n)$. We take the point $\phi(n(e_2+\beta e_1))=\phi(0) +n e_2+n \beta e_1\in \phi(\Gamma_\beta)$ 
and project it on the closest line $\{(0,m)+te_1\mid t\in\R\},m\in \mathbb Z$. Then we find the distance from the projected point to $\Gamma_\alpha$ and denote it as $f(n)$. It is enough to find $n$ such that $f(n)\geq \frac 14$.

Note, that there could not be two closest lines that we consider as if there are two such lines, then the second coordinates of all points in $\phi(\Gamma_\beta)$ is in $\frac12+\mathbb Z$ and $d(\phi(\Gamma_\beta),\Gamma_\alpha)\geq \frac12$.

The value $f(n)$ is the distance from $f(0)+(\beta-\alpha)n$ to $0$ in $\R/\Z$ because when we increase $n$ by $1$, the points in $\Gamma_\alpha$ shift by $\alpha$ horizontally, and the points in $\phi(\Gamma_\beta)$ shift by $\beta$ horizontally but still are given by a shifted integer lattice as long as $n$ is fixed.

Now the existence of such $n$ follows from Lemma \ref{lem:fractional}.
\end{proof}

\begin{proof}[Proof of Theorem \ref{theo:uncountable_family}.]
If $\kappa=1$, it is enough to take $S=\frac 14$, $I=[0,\frac 12]$ and the family of lattices $\{\Gamma_\alpha\}_{\alpha\in I}$ for appropriate $r$ and $R$. For an arbitrary $\kappa>0$, we take the family formed by appropriately scaled lattices.
\end{proof}

\section{Coarse geometry}\label{sec:coarse}

\subsection{Nonembeddability conditions}

Let $\{X_k\}_{k\in\N}$ be a sequence of metric spaces and let $X$ be another metric space. We say that a family of maps $\{i_k\colon X_k\to X\}_{k\in\N}$ is a {\em coarse embedding} if there exist two maps $\rho_-,\rho_+\colon\R_{\geq 0}\to\R_{\geq 0}$ such that $i_k$ is a $(\rho_-,\rho_+)$-coarse embedding for every $k\in\N$.
We say that $X$ {\em contains a coarse disjoint union of $\{X_k\}_{k\in\N}$} if there exists a coarse embedding $\{i_k\colon X_k\to X\}_{k\in\N}$  
such that
$$\dist(i_n(X_n),i_m(X_m))\xrightarrow{m+n\to\infty}\infty.$$

\begin{theorem}[\cite{Laf}]\label{theo:Laf}
There exists a sequence $\{X_k\}_{k\in\N}$ of finite metric spaces such that, if a metric space $X$ contains a coarse disjoint union of $\{X_k\}_{k\in\N}$, then $X$ cannot be coarsely embedded into any uniformly convex Banach space. 
\end{theorem}

To prove Theorem \ref{theo:no_coarse_embedding}, we intend to show that $\PPS_r(\R^d,\kappa)$ and $\PPS^R(\R^d,\kappa)$ contain a coarse disjoint union of any sequence of finite metric spaces. First, we describe a way to embed every finite metric space $X$ into a simpler object in a ``controlled'' manner.

For every $n\in\N\setminus\{0\}$ and $m>0$, endow the space $[0,m]^n$ with the supremum metric $d_m^n$: for every $(x_i)_i,(y_i)_i\in [0,m]^n$, 
\begin{equation}\label{eq:d_m^n}d_m^n((x_i)_i,(y_i)_i)=\max_{i=1,\dots,n}\lvert x_i-y_i\rvert.\end{equation}
For $t>0$ and $n,m\in\N\setminus\{0\}$, we write $C_{t,m}^n$ for the space $([0,tm]\cap t\N)^n$, which is a metric subspace of $([0,tm]^n,d_{tm}^n)$. Thus, $C_{t,m}^n$ can be thought as an $n$-dimensional grid with $m+1$ points along each of $n$ directions with distance $t$ between points in each of $n$ directions equipped with $\ell^\infty$-distance.

\begin{lemma}\label{lemma:emb_from_X_to_cube}
Fix a parameter $t>0$. There exist two maps $\rho_-,\rho_+\colon\R_{\geq 0}\to\R_{\geq 0}$ such that, for every finite metric space $X$, there exist $m,n\in\N$ and a $(\rho_-,\rho_+)$-coarse embedding $i_X\colon X\to C_{t,m}^n$.
\end{lemma}

\begin{proof}
The map $i_X$ is obtained by composing different coarse embeddings whose control functions depend only on the parameter $t$. More precisely, define $n=\lvert X\rvert$ and $m=\lfloor 2\diam X\rfloor$, then we have
$$
\xymatrix{
X\ar[rr]^(.35){i_{K}} \ar@/_2.0pc/[rrrrrr]^{i_{X}}  && [0,2\diam X]^n\ar[rr]^(.6){{\lfloor\cdot\rfloor}^n} && C_{1,m}^n\ar[rr]^{dil(t)} && C_{t,m}^n,
}$$
where the three maps are defined as follows.
\begin{compactenum}[$\bullet$]
\item $i_K$ is the {\em Kuratowski embedding} (see for example \cite{WeiYamZav} where the map was used in a similar context) which is an isometric embedding associating to every point $x\in X$, the value $i_K(x)=(\diam X+d_X(x,z)-d_X(x_0,z))_{z\in X}$ for a fixed $x_0\in X$.
\item $\lfloor\cdot\rfloor^n$ associates to every $(x_i)_{i=1}^n$ the $n$-tuple obtained by applying the floor map to each component separately: $\lfloor\cdot\rfloor^n((x_i)_i)=(\lfloor x_i\rfloor)_i$. It is easy to see that this map is a coarse embedding whose control functions are independent of the parameters. 
\item The map $dil(t)$ associates to each $(x_i)_i\in C_{1,m}^n$ the point $(tx_i)_i\in C_{t,m}^n$. The distance between two images is $t$-times the distance in the original space, and so it is a coarse embedding whose control functions depend only on $t$.
\end{compactenum}
Then $i_X$ is a coarse embedding whose control functions depend only on $t$.
\end{proof}

In the next step in the strategy, we intend to ``slice'' 
$C_{t,m}^n$ along coordinate directions in a uniform way and encode their geometry in appropriate finite subsets of the real line. To do it, we implement a method that has been previously successfully applied in \cite{WeiYamZav} and \cite{Za_GH}. The following lemma is a crucial step in our construction.

\begin{lemma}\label{lemma:WeiYamZav}
Let $m,n\in\mathbb N\setminus\{0\}$. For every parameter $t>0$, there exists a positive number $M=M(t,m,n)$ and an isometric embedding $j_{t,m}^n\colon C_{t,m}^n\to [[0,M]]^{=n}$,  where the target space  (i.e., the family of $n$-point subsets of the interval $[0,M]$) is endowed with the bottleneck distance.
\end{lemma}
\begin{proof}
Let us recall the construction, referring to \cite[Lemma 5.6]{WeiYamZav} for the details. We warn the reader that in \cite[Lemma 5.6]{WeiYamZav}, the proof was carried out for the Hausdorff metric, but it can be applied verbatim to our setting. We define $M=tmn+2tm(n+1)=tm(3n+2)$, and, for every $(x_i)_i\in C_{t,m}^n$, we define
$$j_{t,m}^n((x_i)_i)=\{3tmi+tm(i-1)+x_i\mid i\in\{1,\dots,n\}\}.$$
\end{proof}
We represent  map $j_{t,m}^n$ in Figure \ref{fig:g_m^n}.

\begin{figure}
	\centering
\begin{tikzpicture}
\draw (0,0) node [below] {$0$} -- (8,0) (10,0)--(15,0) node [below] {$M$};
\draw (0,0.2)--(0,0.6) --(2,0.6) node [pos=0.5,above]{$3tm$} -- (2,0.2);
\draw (3,0.2)--(3,0.6) --(5,0.6) node [pos=0.5,above]{$3tm$} -- (5,0.2);
\draw (6,0.2)--(6,0.6) --(8,0.6) node [pos=0.5,above]{$3tm$} -- (8,0.2);
\draw (10,0.2)--(10,0.6) --(12,0.6) node [pos=0.5,above]{$3tm$} -- (12,0.2);
\draw (13,0.2)--(13,0.6) --(15,0.6) node [pos=0.5,above]{$3tm$} -- (15,0.2);
\draw (2,0.8)--(2,1)--(3,1) node[pos=0.5,above]{$tm$}--(3,0.8);
\draw (5,0.8)--(5,1)--(6,1) node[pos=0.5,above]{$tm$}--(6,0.8);
\draw (12,0.8)--(12,1)--(13,1) node[pos=0.5,above]{$tm$}--(13,0.8);

\draw[dashed] (8,0) -- (10,0);
\draw[red] (2,-0.2) -- (2,-0.4) -- (2.7,-0.4) node [pos=0.5,below]{$x_1$}--(2.7,-0.2);
\fill[red] (2.7,0) circle (2pt) (5.3,0) circle (2pt) (12.8,0) circle (2pt);
\draw[red] (5,-0.2) -- (5,-0.4) -- (5.3,-0.4)node [pos=0.5,below]{$x_2$}--(5.3,-0.2);
\draw[red] (12,-0.2) -- (12,-0.4)-- (12.8,-0.4) node [pos=0.5,below]{$x_n$}--(12.8,-0.2);
\fill[blue] (2.3,0) circle (2pt) (5.8,0) circle (2pt) (12.6,0) circle (2pt);
\draw[blue] (2,0.2) -- (2,0.4) -- (2.3,0.4) node [pos=0.5,above]{$y_1$}--(2.3,0.2);
\draw[blue] (5,0.2) -- (5,0.4) -- (5.8,0.4)node [pos=0.5,above]{$y_2$}--(5.8,0.2);
\draw[blue] (12,0.2) -- (12,0.4)-- (12.6,0.4) node [pos=0.5,above]{$y_n$}--(12.6,0.2);
\end{tikzpicture}
\caption{A representation of the subsets $j_{t,m}^n((x_i)_i)$ in red and $j_{t,m}^n((y_i)_i)$ in blue, for $(x_i)_i,(y_i)_i\in C_{t,m}^n$, as defined in Lemma \ref{lemma:WeiYamZav}. The bottleneck distance between the two subsets is given by the maximum of the values $\lvert x_i-y_i\rvert$, where $i\in\{1,\dots,n\}$.}\label{fig:g_m^n}
\end{figure}
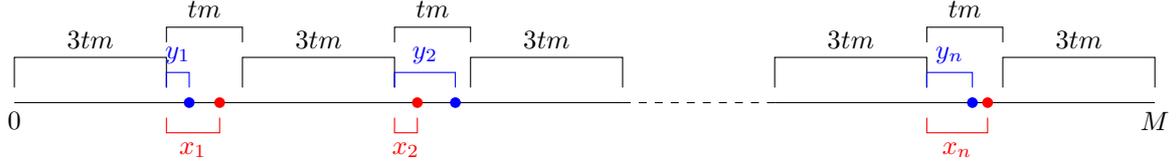

We discuss separately the proofs of Theorem \ref{theo:no_coarse_embedding} for $\PPS_r(\R^d,\kappa)$ and $\PPS^R(\R^d,\kappa)$ in the following two subsequent subsections, \S\ref{sub:theo_no_ce_PPS_r} and \S\ref{sub:theo_no_ce_PPS^R}. In both situations, thanks to Lemma \ref{lemma:emb_from_X_to_cube}, to show that any of the two spaces contains a coarse disjoint union of any family of finite metric spaces, it is enough to prove that it contains a coarse disjoint union of the family $\{C_{t,m}^n\}_{m,n\in\N\setminus\{0\}}$ for an appropriate parameter $t>0$. Clearly, the family is countable, and we want to construct the embeddings by recursion. 

Let us fix a bijection $T\colon(\N\setminus\{0\})^2\to\N\setminus\{0\}$ defined as follows: for every pair $(m,n)\in\N$,
$$T(m,n)=m+\sum_{i=2}^{m+n-1}(i-1)=m+\frac{1}{2}(m+n-2)(m+n-1).$$
In Figure \ref{fig:T} we represent the sequence of 
points $T(m,n)$.
\begin{figure}
	\centering
\begin{tikzpicture}[>=latex]
\draw[help lines] (1,1) grid (5,5);
\draw[thick,->] (1,1)--(2,1);
\draw[thick,->] (2,1)--(1,2);
\draw[thick,->] (1,2)--(3,1);
\draw[thick,->] (3,1)--(2,2);
\draw[thick,->] (2,2)--(1,3);
\draw[thick,->] (1,3)--(4,1);
\draw[thick,->] (4,1)--(3,2);
\draw[thick,->] (3,2)--(2,3);
\draw[thick,->] (2,3)--(1,4);
\draw[thick,->] (1,4)--(5,1);
\draw[thick,->] (5,1)--(4,2);
\draw[thick,->] (4,2)--(3,3);
\draw[thick,->] (3,3)--(2,4);
\draw[thick,->] (2,4)--(1,5);
\end{tikzpicture}
\caption{A representation of the map $T^{-1}\colon\N\setminus\{0\}\to(\N\setminus\{0\})^2$.}\label{fig:T}
\end{figure}
By construction, for every two pairs $(m,n),(m^\prime,n^\prime)\in\N\setminus\{0\}$, 
we have
$m+n\leq m^\prime+n^\prime$ provided that $T(m,n)\leq T(m^\prime,n^\prime)$. For the sake of simplicity, we denote by $\preceq$ the order on $(\N\setminus\{0\})^2$ induced by $T$, i.e., $(m,n)\preceq(m^\prime,n^\prime)$ if $T(m,n)\leq T(m^\prime,n^\prime)$ (and $(m,n)\prec(m^\prime,n^\prime)$ if $T(m,n)<T(m^\prime,n^\prime)$).

\subsection{Theorem \ref{theo:no_coarse_embedding} for $\PPS_r(\R^d,\kappa)$}\label{sub:theo_no_ce_PPS_r}

Let $t=2r$. In this subsection, we want to show that $\PPS_r(\R^d,\kappa)$ contains a coarse disjoint union of the family $\{C_{t,m}^n\}_{n,m}$. As described in the previous subsection, we proceed by recursion to provide a coarse embedding of each element of the family with uniform control functions. Then, we will discuss how to prove that they form a coarse disjoint union. 

We now describe the embedding of $C_{t,m}^n$ depending on two parameters $l$ and $s$ whose role will be discussed later. For a pair of strictly positive real values $l$ and $s$, we define the map $\phi_{t,m}^n\colon C_{t,m}^n\to\PPS_r(\R^d,\kappa)$ by setting
$$\phi_{t,m}^n((x_i)_i)=(P,(s_1e_1,\dots,s_de_d)),$$ 
where $P$ and $s_1,\dots,s_d$ are as follows:
\begin{compactenum}[$\bullet$]
\item $s_1=3m+l+s$, with $s\geq 7M(t,m,n)+3m$, and, for every $i\in\{2,\dots,d\}$, $s_i=i\cdot s_1$, see Lemma \ref{lemma:WeiYamZav} for the definition of $M(t,m,n)$)
\item $P$ is the union of three components: 
\begin{compactenum}[-]
\item $X_{t,m}^n=3m(e_1+\dots+e_d)+Q(l)$, where $Q(l)$ is the densest set of points in a $d$-dimensional parallelepiped of edges $l,2l,\dots,dl$, 
compatibly with the packing radius $r$, 
\item $Y_{t,m}^n$, the union of portions of the $(d-1)$-dimensional lattice $(t\Z)^{d-1}$ covering the facets of the unit cell, i.e., the subsets $\{(x_1,\dots,x_{i-1},0,x_{i+1},\dots,x_d)\mid\forall j\in\{1,\dots,d\}\setminus\{i\},\, 0\leq x_j< s_j\}$ for $i\in\{1,\dots,d\},$
\item and $Z_{t,m}^n((x_i)_i)=(3m+l+3M,3m,\dots,3m)+j_{t,m}^n((x_i)_i)\cdot e_1$, the {\em signature} of the point $(x_i)_i$.
\end{compactenum}
\end{compactenum}

Note that the two sets of points $X_{t,m}^n$ and $Y_{t,m}^n$ are portions of the motif that are common to all the images of the map $\phi_{t,m}^n$. Only $Z_{t,m}^n((x_i)_i)$ retains information about the point $(x_i)_i$.

In order to attain density $\kappa$, the parameters $l$ and $s$ are tightly connected: once one is fixed, the other one has to be chosen accordingly. Note that the points $Y_{t,m}^n$ do not prevent the periodic point set from achieving the desired density since the $d$-dimensional volume of the empty portion of the interior of the unit cell grows faster than the sum of the $(d-1)$-dimensional volumes of its facets. The unit cell of $\phi_{t,m}^n((x_i)_i)$ is represented in Figure \ref{fig:phi_tm_n_x_i}.

\begin{figure}
\centering
\begin{tikzpicture}[>=latex]
\tikzset{decorate sep/.style 2 args={decorate,decoration={shape backgrounds,shape=circle,shape size=#1,shape sep=#2}}}

\fill [white!90!black] (0,0) -- (8.5,0) -- (8.5,0.5) -- (0,0.5) -- (0,0);
\fill [white!90!black] (0,0.5) -- (0.5,0.5) -- (0.5,7.5) -- (0,7.5) -- (0,0.5);
\fill [white!90!black] (0.5,7.5) -- (0.5,7) -- (8.5,7) -- (8.5,7.5) -- (0.5,7.5);
\fill [white!90!black] (8,0.5) -- (8.5,0.5) -- (8.5,7) -- (8,7) -- (8,0.5);

\fill [white!60!red] (5,0.5) -- (5,0.65) -- (6.5,0.65) -- (6.5,0.5) -- (5,0.5);
\fill [white!60!red] (0.5,0.5) -- (4,0.5) -- (4,4) -- (0.5,4) -- (0.5,0.5);
\draw[-] (0,0) -- (8.5,0);
\draw[-] (0,0) -- (0,7.5);
\draw (0,7.7) -- (0,7.9) -- (8.5,7.9) node [pos=0.5,above]{$s_1$} -- (8.5,7.7);
\draw [dashed] (8.5,0) -- (8.5,7.5) -- (0,7.5);
\draw (-0.5,0) -- (-1.1,0) -- (-1.1,7.5) node [pos=0.5,left]{$s_2=2(3m+l+s)$} -- (-0.5,7.5);
\draw (-0.2,7) -- (-0.4,7) -- (-0.4,7.5) node [pos=0.5,left]{$3m$} -- (-0.2,7.5);
\draw (-0.2,0) -- (-0.4,0) -- (-0.4,0.5) node [pos=0.5,left]{$3m$} -- (-0.2,0.5);
\draw (4,-0.2) -- (4,-0.4) -- (5,-0.4) node [pos=0.5,below]{$3M$} -- (5,-0.2);
\draw (5,-0.5) -- (5,-0.7) -- (6.5,-0.7) node [pos=0.5,above]{$M$} -- (6.5,-0.5);
\draw (6.5,-0.2) -- (6.5,-0.4) -- (7.5,-0.4) node [pos=0.5,below]{$3M$} -- (7.5,-0.2);
\draw (8,-0.2) -- (8,-0.4) -- (8.5,-0.4) node [pos=0.5,below]{$3m$} -- (8.5,-0.2);
\draw (4,-0.8) -- (4,-1) -- (8.5,-1) node [pos=0.5,below]{$s$} -- (8.5,-0.8);
\draw (5.75,0.65) node[above]{$j_{t,m}^n((x_i)_i)$};
\draw [dashed] (6.5,0) -- (6.5,0.5);
\draw [dashed] (0,4) -- (8.5,4);
\draw [dashed] (5,0) -- (5,0.5);
\draw [dashed] (4,0)--(4,7.5) (0,0.5)--(8.5,0.5) (0,7)--(8.5,7) (0.5,0)--(0.5,7.5) (8,0)--(8,7.5);
\draw (2.25,2.25) node  {$Q(l)$};
\draw (-0.5,0.5) -- (-0.7,0.5) -- (-0.7,4) node [pos=0.5,right]{$2l$} -- (-0.5,4);
\draw (0,-0.2) -- (0,-0.4) -- (0.5,-0.4) node [pos=0.5,below]{$3m$} -- (0.5,-0.2);
\draw (0.5,-0.5) -- (0.5,-0.7) -- (4,-0.7) node [pos=0.5,above]{$l$} -- (4,-0.5);
\draw[decorate sep={1mm}{4mm},fill,white!60!red] (0,0) -- (8.5,0) (0,0)--(0,7.5);
\end{tikzpicture}
\caption{A (not to scale) representation of the unit cell of $\phi_{t,m}^n((x_i)_i)\in\PPS_r(\R^2,\kappa)$. The motif consists of the red dots on the boundary of the unit cell, and of points contained in the two other red areas: the rectangle  $Q(l)$, and the signature points $j_{t,m}^n((x_i)_i)$. The gray area is empty, and the same holds for a rectangle at the top right (i.e., $[3m+l,3m+l+s]\times[3m+2l,s_2]$). The latter provides a lower bound $s/2$ to the covering radius of the induced periodic point set. Similarly, $(3m+l+s+2r\sqrt{d-1})/2$ is an upper bound (the last addendum is half of the diagonal of a $(d-1)$-dimensional cube with edge length equal to $2r$). These bounds are formalized by Lemma \ref{lemma:c_bounds_for_phi}. }\label{fig:phi_tm_n_x_i}
\end{figure}

Let us provide some immediate bounds on the covering radius of periodic point sets $\phi_{t,m}^n((x_i)_i)$.

\begin{lemma}\label{lemma:c_bounds_for_phi}
In the previous notation, for every $(x_i)_i\in C_{t,m}^n,$
$$\frac{s}{2}\leq c(\phi_{t,m}^n((x_i)_i))\leq \frac{s_1+2r\sqrt{d-1}}{2}.$$
\end{lemma}
\begin{proof}
Since the parallelepiped 
$$[3m+l,3m+l+s]\times[3m+2l,s_2]\times\cdots\times[3m+dl,s_d]$$ 
(the top right rectangle in Figure \ref{fig:phi_tm_n_x_i}) in the unit cell is empty, its center is at least $s/2$-far apart from any point of the periodic point set.

As for the upper bound, every point in $\R^d$ is at distance at most $s_1/2$ from the boundary of a copy of the unit cell. Then, a point in the corresponding copy of $Y_{t,m}^n$ can be found at distance at most the diameter of a $(d-1)$-dimensional cube of edge length $2r$, which is $2r\sqrt{d-1}$. 
\end{proof}

\begin{lemma}\label{lemma:phi_tm_n_ce}
In the previous notation, for sufficiently large and appropriate $l$ and $s$, $\phi_{t,m}^n$ is a coarse embedding whose control functions do not depend on $m$ and $n$. More precisely:
\begin{compactenum}[(a)]
\item if $\PPS_r(\R^d,\kappa)$ is endowed with $d_B$, then $\phi_{t,m}^n$ is an isometric embedding;
\item if $\PPS_r(\R^d,\kappa)$ is endowed with $\EB$, then $\phi_{t,m}^n$ is a $(\rho_-,\rho_+)$-coarse embedding, where $\rho_-\colon x\mapsto x/2$ and $\rho_+=id$.
\end{compactenum}
\end{lemma}
\begin{proof}
The map $\phi_{t,m}^n$ is well-defined (i.e., the periodic point set has packing radius at least $r$) thanks to the construction of $j_{t,m}^n$ and the constraints on $Q(l)$ and $Y_{t,m}^n$. 

Item (a) easily follows from Lemma \ref{lemma:WeiYamZav}. In fact, we already observed that the subsets $X_{t,m}^n$ and $Y_{t,m}^n$ are common to the images of any pair of points in $C_{t,m}^n$, and therefore they can be matched. Two different signatures can be matched as described in Lemma \ref{lemma:WeiYamZav}.

Let us now discuss item (b). The fact that $\rho_+=id$ is a control function is easy to check. In fact, $\EB$ is always bounded from above by $d_B$ and item (a) implies the claim. As for $\rho_-$, let us divide the proof into steps. 

Let us take two points $(x_i)_i,(y_i)_i\in C_{t,m}^n$. Clearly, $d_{EH}(\phi_{t,m}^n((x_i)_i),\phi_{t,m}^n((y_i)_i))\leq m=\diam(C_{t,m}^n)$.  Let us fix an isometry $\Psi$ of $\R^d$ realizing the Euclidean bottleneck distance.

The points of the copies of $Y_{t,m}^n$ are at least $3m$ far apart from any other point of the periodic point sets (the gray frame in Figure \ref{fig:phi_tm_n_x_i} surround those points).  Furthermore, the lengths of the edges of the unit cell are very different for each dimension. Therefore, there cannot be any rotation involved in $\Psi$, and the linear part of $\Psi$ has to send each vector $e_i$ into either $e_i$ or $-e_i$.

Since the densest part of the unit cell, i.e., $X_{t,m}^n$, is placed close to a specific corner of the unit cell, and far from all the others, the linear part of $\Psi$ must be the identity map. Hence, $\Psi$ is a translation.

Note that any translation, even though it may lower the distance between the signatures of the two points ($Z_{t,m}^n((x_i)_i)$ and $Z_{t,m}^n((y_i)_i)$) creates a misalignment between the two copies of $X_{t,m}^n$ and $Y_{t,m}^n$. Hence, the best trade-off is achieved at most at $d_B(\phi_{t,m}^n((x_i)_i),\phi_{t,m}^n((y_i)_i))/2$, and the claim follows thanks to item (a).
\end{proof}

The next step consists in checking whether the coarse embedding defined on each $C_{t,m}^n$ separately shows that $\PPS_r(\R^d,\kappa)$ contains their coarse disjoint union. In order to do it, we need to evaluate the distance between $\phi_{t,m}^n(C_{t,m}^n)$ and $\phi_{t,m^{\prime}}^{n^\prime}(C_{t,m^{\prime}}^{n^\prime})$ for distinct pairs $(m,n)$ and $(m^\prime,n^\prime)$.

\begin{lemma}
\label{lemma:c_implies_bounds_on_dB}
Let $X,Y\in\PPS(\R^d,\kappa)$. Suppose that $c(X)\leq R$ and $c(Y)>S$ for some $R,S\in\R_{\geq 0}$ with $R<S$. Then $d_B(X,Y)\geq\EB(X,Y)\geq S-R$.
\end{lemma}
\begin{proof}
Since $c(Y)>S$, there exists $y\in\R^d$ such that $\overline{B(y,S)}\cap Y=\emptyset$. Let us fix an isometry $f\in\Iso(\R^d)$. It is clear that $c(X)=c(f(X))$. Since $c(X)\leq R$, there exists $x\in X$ such that $\lvert\lvert f(x)-y\rvert\rvert\leq R$. Then $d_B(f(X),Y)\geq S-R$ since $B(f(x),S-R)\cap Y\subseteq B(y,S-R+R)\cap Y=\emptyset$.
\end{proof}

In order to use Lemma \ref{lemma:c_implies_bounds_on_dB}, we impose further constraints on the constants $l$ and $s$ depending on the pair of indices $(m,n)$ that we are considering.  For this reason, let us rename the two constants associated with the map $\phi_{t,m}^n$ by $l(m,n)$ and $s(m,n)$, respectively. 

We define additional inequalities following the order $\preceq$ on $(\N\setminus\{0\})^2$. We require that, for a given $(m,n)\in(\N\setminus\{0\})^2$, 
\begin{equation}\label{eq:constraint_for_dist}\frac{3m+l(m,n)+s(m,n)+2r\sqrt{d-1}}{2}>\sup_{(m^\prime,n^\prime)\prec(m,n)}\Big(\frac{s(m^\prime,n^\prime)}{2}\Big)+2^{T(m,n)}.\end{equation}
By choosing $s(m,n)$ and $l(m,n)$ recursively and large enough we can guarantee the inequality \eqref{eq:constraint_for_dist}.

We have all the ingredients to prove the first part of Theorem \ref{theo:no_coarse_embedding} regarding the packing radius.

\begin{proof}[Proof of Theorem \ref{theo:no_coarse_embedding} for $\PPS_r(\R^d,\kappa)$]
The maps $\{\phi_{t,m}^n\colon C_{t,m}^n\to\PPS_r(\R^d,\kappa)\}_{m,n\in\N\setminus\{0\}}$ define a coarse embedding according to Lemma \ref{lemma:phi_tm_n_ce}. Furthermore, by carefully choosing the constants $l(m,n)$ and $s(m,n)$, Lemmas \ref{lemma:c_implies_bounds_on_dB}, \ref{lemma:c_bounds_for_phi} and \eqref{eq:constraint_for_dist} imply that 
the images $\phi_{t,m}^n(C_{t,m}^n)$ are scattered in the space of periodic point sets. In fact, 
$$\dist(\phi_{t,m}^n(C_{t,m}^n),\phi_{t,m^\prime}^{n^\prime}(C_{t,m^\prime}^{n^\prime}))\geq 2^{\max\{T(m,n),T(m^\prime,n^\prime)\}},$$
which diverges. Therefore, 
$\PPS_r(\R^d,\kappa)$ contains a coarse disjoint union of the family $\{C_{t,m}^n\}_{m,n\in\N\setminus\{0\}}$. Thanks to Lemma \ref{lemma:emb_from_X_to_cube}, $\PPS_r(\R^d,\kappa)$ contains a coarse disjoint union of any family of finite metric spaces, and so Theorem \ref{theo:Laf} implies the desired statement.
\end{proof}

\subsection{Theorem \ref{theo:no_coarse_embedding} for $\PPS^R(\R^d,\kappa)$}\label{sub:theo_no_ce_PPS^R}

In this subsection, we focus on proving the second claim of Theorem \ref{theo:no_coarse_embedding}, namely that there is no coarse embedding of $\PPS^R(\R^d,\kappa)$ into a uniformly convex Banach space. The strategy of the proof is the same as in the proof of the first part of the said theorem, namely to embed the coarse disjoint union of the family $\{C_{t,m}^n\}_{m,n\in\N\setminus\{0\}}$ into the space. The 
embeddings described in \S\ref{sub:theo_no_ce_PPS_r} have no uniform upper bound on the covering radius. Therefore, we have to develop a different approach.

First, let us construct a different embedding of $C_{m}^n=C_{1,m}^n$ for every $m,n\in\N\setminus\{0\}$. Let $\varepsilon\in(0,1)$, and for every positive natural number $N$, denote by 
$$I(N,\varepsilon)=\Big\{\varepsilon\frac{i}{N}\mid i\in\{0,\dots,N-1\}\Big\}.$$
If, furthermore, $L,K$ are positive natural numbers, we define, for every $(x_i)_i\in C_m^n$,
\begin{equation}\label{eq:J_m^n}\begin{multlined}J_m^n((x_i)_i,\varepsilon,L,K)=\bigcup_{i=1}^n(3(m+1)i+(m+1)(i-1)+x_i+I(L,\varepsilon))\cup\\
\cup(3(n+1)(m+1)+n(m+1)+I(K,\varepsilon)).\end{multlined}\end{equation}
We define 
$$M(m,n,\varepsilon)=3(n+1)(m+1)+n(m+1)+\varepsilon.$$

We represent subset $J_m^n((x_i)_i,\varepsilon,L,K)$ in Figure \ref{fig:J_m^n}.

Given $j\in\{1,\dots,n\}$, let us call $A_j((x_i)_i,m,n)$ the subset of $J_m^n((x_i)_i,\varepsilon,L,K)$ of cardinality 
$L$ representing the value $x_j$, 
i.e., the subset $3(m+1)i+(m+1)(j-1)+x_j+I(L,\varepsilon)$. 
Furthermore, let $A_{n+1}((x_i)_i,m,n)$ be the remaining subset of $K$-many points. 
For the sake of simplicity, when $m$ and $n$ are clear from the context, we simply write $A_j((x_i)_i)$ instead of $A_j((x_i)_i,m,n)$.

\begin{figure}
	\centering
\begin{tikzpicture}
\draw (0,0) node [left] {$0$} -- (8,0) (10,0)--(15.3,0) node [right] {$M(m,n,\varepsilon)$};
\draw (0,0.2)--(0,0.6) --(2,0.6) node [pos=0.5,above]{$3(m+1)$} -- (2,0.2);
\draw (3,0.2)--(3,0.6) --(5,0.6) node [pos=0.5,above]{$3(m+1)$} -- (5,0.2);
\draw (6,0.2)--(6,0.6) --(8,0.6) node [pos=0.5,above]{$3(m+1)$} -- (8,0.2);
\draw (10,0.2)--(10,0.6) --(12,0.6) node [pos=0.5,above]{$3(m+1)$} -- (12,0.2);
\draw (13,0.2)--(13,0.6) --(15,0.6) node [pos=0.5,above]{$3(m+1)$} -- (15,0.2);
\draw (2,0.8)--(2,1)--(3,1) node[pos=0.5,above]{$m+1$}--(3,0.8);
\draw (5,0.8)--(5,1)--(6,1) node[pos=0.5,above]{$m+1$}--(6,0.8);
\draw (12,0.8)--(12,1)--(13,1) node[pos=0.5,above]{$m+1$}--(13,0.8);

\draw[dashed] (8,0) -- (10,0);
\draw (2,-0.2) -- (2,-0.4) -- (2.6,-0.4) node [pos=0.5,below]{$x_1$}--(2.6,-0.2);
\fill[red] (2.6,0.07)--(2.6,-0.07)--(2.9,-0.07)--(2.9,0.07)--(2.6,0.07); 
\draw (2.6,0.2)--(2.6,0.4)--(2.9,0.4)node[pos=0.5,above]{$\varepsilon$}--(2.9,0.2);
\draw[->,red] (2.75,-0.75)node[below]{$A_1((x_i)_i)$}--(2.75,-0.2);
\fill[red] (5.3,0.07)--(5.3,-0.07)--(5.6,-0.07)--(5.6,0.07)--(5.3,0.07);
\draw (5.3,0.2)--(5.3,0.4)--(5.6,0.4)node[pos=0.5,above]{$\varepsilon$}--(5.6,0.2);
\draw[->,red] (5.45,-0.75)node[below]{$A_2((x_i)_i)$}--(5.45,-0.2);
\fill[red] (12.65,0.07)--(12.65,-0.07)--(12.95,-0.07)--(12.95,0.07)--(12.65,0.07);
\draw (12.65,0.2)--(12.65,0.4)--(12.95,0.4)node[pos=0.5,above]{$\varepsilon$}--(12.95,0.2);
\draw[->,red] (12.8,-0.75)node[below]{$A_n((x_i)_i)$}--(12.8,-0.2);
\fill[blue] (15,-0.07)--(15,0.07)--(15.3,0.07)--(15.3,-0.07)--(15,-0.07);
\draw (15,0.2)--(15,0.4)--(15.3,0.4)node[pos=0.5,above]{$\varepsilon$}--(15.3,0.2);
\draw[->,blue] (15.15,-0.75)node[below]{$A_{n+1}((x_i)_i)$}--(15.15,-0.2);
\draw (5,-0.2) -- (5,-0.4) -- (5.3,-0.4)node [pos=0.5,below]{$x_2$}--(5.3,-0.2);
\draw (12,-0.2) -- (12,-0.4)-- (12.65,-0.4) node [pos=0.5,below]{$x_n$}--(12.65,-0.2);
\end{tikzpicture}
\caption{A (not to scale) representation of the subset $J_{m}^n((x_i)_i,\varepsilon,L,K)$ for $(x_i)_i\in C_{t,m}^n$, as defined in \eqref{eq:J_m^n}. In red, the subsets $A_j((x_i)_i)$ with $j\in\{1,\dots,n\}$ of $L$ points, while in blue $A_{n+1}((x_i)_i)$ consisting of $K$ points.}\label{fig:J_m^n}
\end{figure}
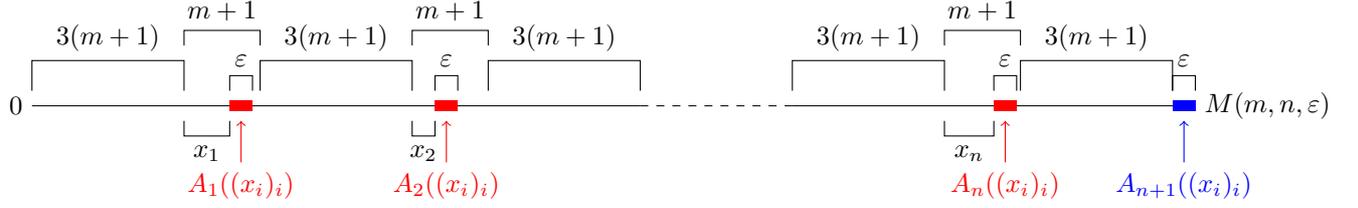

We now construct a coarse embedding $\psi_m^n\colon C_m^n\to\PPS^R(\R^d,\kappa)$. As in the previous section, we define it depending on a family of parameters that will be adjusted later on to achieve the desired properties. Let us fix two natural numbers $L,K>0$, a positive constant $\varepsilon\in (0,1)$ and a real number $l> 3M(m,n,\varepsilon)$. For the sake of simplicity, if $m$, $n$ and $\varepsilon$ are clear, we simply write $M$ for $M(m,n,\varepsilon)$. Let 
$P(l)\subseteq [0,l]^d$ 
be the least dense set of points provided that $\overline{ B(P(l),R)}\supseteq[0,l]^d$ (i.e., $P(l)$ $R$-covers the $d$-dimensional cube of $[0,l]^d$). 

Assume now that $d\geq 2$, and we consider the case $d=1$ separately. Fix, for every $i\in\{2,\dots,d\}$, a real number $s_i \in [3M, l]$, such that the segment $[(0,s_2,\dots,s_d),(l,s_2,\dots,s_d)]$ does not contain points of $P(l)$ (they exists for cardinality reasons). For $(x_i)_i\in C_m^n$, we define
\begin{gather*}
\psi_m^n((x_i)_i)=(Q,(l\Z)^d),\quad\text{where}\\
Q=P(l)\cup((3M,s_2,\dots,s_d)+J_m^n((x_i)_i,\varepsilon,L,K)\cdot e_1).
\end{gather*}
We represent its unit cell in Figure \ref{fig:my_label}. Note that, thanks to the construction of $P(l)$, the periodic point set just defined has covering radius at most $R$.

\begin{figure}
    \centering
\begin{tikzpicture}
    \draw[pattern=my crosshatch dots]  (0,0) rectangle (5,5);
    \draw (-0.2,0) -- (-0.4,0) -- (-0.4,1) node [pos=0.5,left]{$3M$} -- (-0.2,1);
    \draw (-0.5,0) -- (-1.3,0) -- (-1.3,1.475) node [pos=0.5,left]{$s_2$} -- (-0.2,1.475);

    \draw (0,-0.2) -- (0,-0.4) -- (1,-0.4) node [pos=0.5,below]{$3M$} -- (1,-0.2);
    \draw (1,-0.5) -- (1,-0.7) -- (3,-0.7) node [pos=0.5,above]{$M$} -- (3,-0.5);
    \draw[dashed] (0,1)--(5,1) (1,0)--(1,5) (3,0)--(3,1.55)
    (0,1.475)--(5,1.475);
    \fill [white!60!red] (1,1.40) -- (1,1.55) -- (3,1.55) -- (3,1.4) -- (1,1.4);
    \draw (2.5,1.7) node[above]{$J_m^n((x_i)_i,\varepsilon,L,K)$};
    \draw[very thick,->] (0,0)--(5,0)node[below]{$le_1$};
    \draw[very thick,->] (0,0)--(0,5)node[left]{$le_2$};
\end{tikzpicture}    \caption{A (not to scale) representation of the unit cell of $\psi_m^n((x_i)_i,\varepsilon,L,K)$, where $M$ stands for $M(m,n,\varepsilon)$. The red rectangle contains the points of $J_m^n((x_i)_i,\varepsilon,L,K)$ encoding the information about the point $(x_i)_i$.}
\label{fig:my_label}
\end{figure}
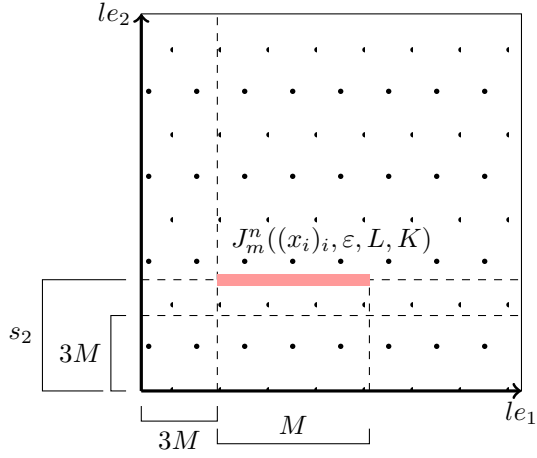

As for the case $d=1$, by choosing an appropriate $s_1\geq 3M$ and possibly increasing the space between consecutive $A_j((x_i)_i,m,n)$, we may assume that, for every $x\in C_m^1$, $s_1+J_m^n(x,\varepsilon,L,K)$ does not intersect $P(l)$. Then, define $\psi_m^1(x)=(P(l)\cup(s_1+J_m^n(x,\varepsilon,L,K)),l\Z)$.

Before proving that, for suitable constants, $\psi_m^n$ is a coarse embedding, we state an immediate, but crucial result.
\begin{lemma}\label{lemma:too_many_points}
Let $X$ be a Delone set in $\R^d$, $m\in\N$, and 
$$N=\max_{x\in\R^d}\lvert B(x,m+1)\cap X\rvert+1.$$
If $Y$ is a subset of $\R^d$ such that $\lvert Y\rvert=N$ and $\diam Y<1$, then, for every injection $f\colon Y\to X$, there exists $y\in Y$ satisfying $\lvert\lvert y-f(y)\rvert\rvert>m$.
\end{lemma}
\begin{proof}
Suppose, by contradiction, that there is an injection $f$ such that, for every $y\in Y$, $\lvert\lvert y-f(y)\rvert\rvert\leq m$. 
If $x\in Y$, for every $y\in Y$, $\lvert\lvert x-y\rvert\rvert<1$, and so
$$\lvert\lvert x-f(y)\rvert\rvert\leq\lvert\lvert x-y\rvert\rvert +\lvert\lvert y-f(y)\rvert\rvert<1+m.$$
Therefore, the ball $B(x,m+1)$ contains at least $N$-many points of $X$, which contradicts the definition of $N$.
\end{proof}

\begin{lemma}\label{lemma:PPS^R_ce}
In the previous notation and for suitable choices of the parameters $L,K\in\N\setminus\{0\}$ and $l\in\R_{>0}$, 
$\psi_{m}^n$ is a coarse embedding whose control functions do not depend on $m$ and $n$. More precisely:
\begin{compactenum}[(a)]
\item if $\PPS^R(\R^d,\kappa)$ is endowed with $d_B$, then $\psi_{m}^n$ is a $(\rho_-,\rho_+)$-coarse embedding, where $$\rho_-\colon x\mapsto \max\{x-\varepsilon,0\} \quad\text{ and }\quad\rho_+=id;$$
\item if $\PPS^R(\R^d,\kappa)$ is endowed with $\EB$, then $\psi_{m}^n$ is a $(\rho_-,\rho_+)$-coarse embedding, where $$\rho_-\colon x\mapsto \max\{x/2-\varepsilon,0\} \quad\text{ and }\quad \rho_+=id.$$
\end{compactenum}
\end{lemma}
\begin{proof}
We prove (b), whereas (a) can be deduced using some of the ideas here described. More explicitly, we show that, for every $(x_i)_i,(y_i)_i\in C_m^n$,
\begin{equation}\label{eq:EB_PPS^R}\frac{d_m^n((x_i)_i,(y_i)_i)}{2}-\varepsilon\leq \EB(\psi_m^n((x_i)_i),\psi_m^n((y_i)_i))\leq d_m^n((x_i)_i,(y_i)_i)\end{equation}
(see \eqref{eq:d_m^n} for the definition of the metric $d_m^n$). The upper bound of \eqref{eq:EB_PPS^R} always holds independently of the parameters. In fact, already \begin{equation}\label{eq:d_B_upper_bound}d_B(\psi_m^n((x_i)_i),\psi_m^n((y_i)_i))\leq d_m^n((x_i)_i,(y_i)_i)\end{equation} 
since the most immediate  matching already achieves the value $d_m^n((x_i)_i,(y_i)_i)$. Indeed, consider the matching associating all the points in $P(l)$ to themselves, and, for every $j=1,\dots,n+1$, $A_j((x_i)_i)$ to $A_j((y_i)_i)$ (see Figure \ref{fig:upper_bound}).
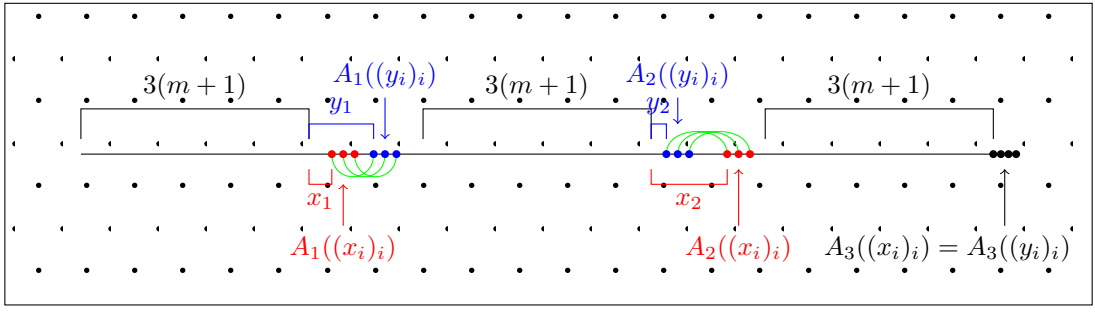
\begin{figure}
	\centering
\begin{tikzpicture}
 \draw[pattern=my crosshatch dots]  (-1,-2) rectangle (13.3,2);

\draw (0,0) 
--(12.3,0) 
;
\draw (0,0.2)--(0,0.6) --(3,0.6) node [pos=0.5,above]{$3(m+1)$} -- (3,0.2);
\draw (4.5,0.2)--(4.5,0.6) --(7.5,0.6) node [pos=0.5,above]{$3(m+1)$} -- (7.5,0.2);
\draw (9,0.2)--(9,0.6) --(12,0.6) node [pos=0.5,above]{$3(m+1)$} -- (12,0.2);
\draw[green]  (3.3,0) to[out=-90,in=-90,distance=0.4cm] (3.85,0) (3.45,0) to[out=-90,in=-90,distance=0.4cm] (4,0) (3.6,0) to[out=-90,in=-90,distance=0.4cm] (4.15,0);

\draw[red] (3,-0.2) -- (3,-0.4) -- (3.3,-0.4) node [pos=0.5,below]{$x_1$}--(3.3,-0.2);
\fill[red] (3.6,0) circle (1.5pt) (3.45,0) circle (1.5pt) (3.3,0) circle (1.5pt);
\fill[blue] (3.85,0) circle (1.5pt) (4,0) circle (1.5pt) (4.15,0) circle (1.5pt);
\draw[blue] (3,0.2) -- (3,0.4) -- (3.85,0.4) node [pos=0.5,above]{$y_1$}--(3.85,0.2);
\draw[green]  (8.5,0) to[out=90,in=90,distance=0.4cm] (7.7,0) (8.65,0) to[out=90,in=90,distance=0.4cm] (7.85,0) (8.8,0) to[out=90,in=90,distance=0.4cm] (8,0);

\fill[red] (8.5,0) circle (1.5pt) (8.65,0) circle (1.5pt) (8.8,0) circle (1.5pt);
\fill[blue] (7.7,0) circle (1.5pt) (7.85,0) circle (1.5pt) (8,0) circle (1.5pt);
\draw[blue] (7.5,0.2) -- (7.5,0.4) -- (7.7,0.4) node [pos=0.5,above]{$y_2$}--(7.7,0.2);

\draw[->,red] (3.45,-0.95)node[below]{$A_1((x_i)_i)$}--(3.45,-0.4);
\draw[->,blue] (4,0.75)node[above]{$A_1((y_i)_i)$}--(4,0.2);
\draw[->,blue] (7.85,0.75)node[above]{$A_2((y_i)_i)$}--(7.85,0.4);

\draw[->,red] (8.65,-0.95)node[below]{$A_2((x_i)_i)$}--(8.65,-0.2);
\fill (12,0) circle (1.5pt) (12.1,0) circle (1.5pt) (12.2,0) circle (1.5pt) (12.3,0) circle (1.5pt);
\draw[->] (12.15,-0.95)--(12.15,-0.2);
\draw (11.4,-0.95)node[below]{$A_{3}((x_i)_i)=A_3((y_i)_i)$};
\draw[red] (7.5,-0.2) -- (7.5,-0.4) -- (8.5,-0.4)node [pos=0.5,below]{$x_2$}--(8.5,-0.2);
\end{tikzpicture}
\caption{A matching showing the inequality \eqref{eq:d_B_upper_bound} for $n=2$. Portions of the periodic point sets $\psi_m^2((x_i)_i)$ and $\psi_m^2((y_i)_i)$ are represented as the union of the red and the black dots and of the blue and the black dots, respectively. Then, we consider the matching associating the black dots with themselves and the red dots with the blue ones according to the green lines. By construction, this matching associates points that are at most $\max_{i=1,2}\lvert x_i-y_i\rvert=d_m^2((x_i)_i,(y_i)_i)$-far apart.}\label{fig:upper_bound}
\end{figure}

Let $X$ be the periodic point set having $P(l)$ as motif and $(l\Z)^d$ as lattice. Define 
\begin{equation}\label{eq:L_def}L=\max_{x\in\R^d}\lvert B(x,m+1)\cap X\rvert+1\end{equation}
and choose $K\geq 2L$.

Let $\Phi\in\Iso(\R^d)$ such that $$\EB(\psi_m^n((x_i)_i),\psi_m^n((y_i)_i))=d_B(\psi_m^n((x_i)_i),\Phi(\psi_m^n((y_i)_i))).$$
In particular, $d_B(\psi_m^n((x_i)_i),\Phi(\psi_m^n((y_i)_i)))\leq m$. 
Let $f$ be the matching between $\psi_m^n((x_i)_i)$ and $\Phi(\psi_m^n((y_i)_i))$ such that $\sup_{x\in \psi_m^n((x_i)_i)}\lvert\lvert x-f(x)\rvert\rvert$ realizes the bottleneck distance (see Proposition \ref{PropDbMIN}). 

In the next part of the proof, we show that $\Phi$ can be assumed to be a translation without loss of generality. In Figure \ref{fig:lower_bound}, we represent the steps to achieve this result.

\begin{claim}\label{claim:matching_n+1}
$f$ matches at least one point of a copy of $A_{n+1}((x_i)_i)$ to one of $\Phi(A_{n+1}((y_i)_i))$.
\end{claim}
\begin{proof}[Proof of Claim \ref{claim:matching_n+1}]
Let us consider the union of balls of radius $m$ around 
the points of $A_{n+1}((x_i)_i)$. Since $K\geq L$ and $\diam A_{n+1}((x_i)_i)$, Lemma \ref{lemma:too_many_points} implies that $f$ cannot match all the $K$ points to points of $X$, and so there exists $j\in\{1,\dots,n+1\}$ such that $\Phi(A_{j}((y_i)_i)$ has non-empty intersection with $B\big(A_{n+1}((x_i)_i),m\big)$. Note that $j$ is unique since the distance between two of those blocks is larger than $3(m+1)$.

If $j\neq n+1$, $f$ can match at most $L$ points of $A_{n+1}((x_i)_i)$ to $\Phi(A_j((y_i)_i))$. Thus, $A_{n+1}((x_i)_i)$ would still have at least $L$ points unmatched ($K\geq 2L$). Using Lemma \ref{lemma:too_many_points} again, we can show that the remaining points cannot be matched with $\Phi(X)$. Therefore, $j=n+1$.
\end{proof}
\begin{claim}\label{claim:matching}
For every $j\in\{1,\dots,n+1\}$, $f$ matches at least one point of a copy of $A_{j}((x_i)_i)$ to one of $\Phi(A_{j}((y_i)_i))$.
\end{claim}
\begin{proof}[Proof of Claim \ref{claim:matching}]
Claim \ref{claim:matching_n+1} implies the result for $j=n+1$. Consider now $A_n((x_i)_i)$. These $L$ points cannot be matched with $\Phi(X)$ because of Lemma \ref{lemma:too_many_points}. By construction, since $A_{n+1}((x_i)_i)$ is at distance at most $m$ from $\Phi(A_{n+1}((y_i)_i))$, the only block that can be at distance at most $m$ from $A_n((x_i)_i)$ is $\Phi(A_n((y_i)_i))$. Therefore, the claim is verified also for $j=n$. Iterating this procedure, the claim follows.
\end{proof}

As a consequence of Claim \ref{claim:matching}, we have that the linear part of $\Phi$ has to preserve the vector $e_1$. Therefore, on the subsets $A_j((y_i)_i)$ and their copies $\Phi$ acts as a translation by a vector $v$. Without loss of generality, we can assume that $v=ke_1$ (otherwise, take its projection onto the line generated by $e_1$) and $k\in[0,l)$. 

\begin{figure}
    \centering
  \begin{subfigure}[b]{0.4\textwidth}
\centering \begin{tikzpicture}
 \draw[pattern=my crosshatch dots]  (0,1) rectangle (5.5,5);
 \draw[lightgray] (0,1.8)--(5.5,1.8);
 \fill[blue] (4,1.75)--(4.2,1.75)--(4.2,1.85)--(4,1.85)--(4,1.75);
 \draw[blue] (4.1,1.85) node[below]{$A_{j+1}((x_i)_i)$};
 \fill[blue] (1.5,1.75)--(1.7,1.75)--(1.7,1.85)--(1.5,1.85)--(1.5,1.75);
\draw[blue] (1.6,1.85) node[below]{$A_{j}((x_i)_i)$};
 \begin{scope}[rotate=-30,shift={(-2,2.5)}]
   \draw[white!60!red] (4,1.8) circle (17pt);
     \draw[lightgray] (-0.48,1.8)--(4,1.8);
 \fill[red] (4,1.75)--(4.2,1.75)--(4.2,1.85)--(4,1.85)--(4,1.75);
  \draw[red] (4.1,1.8) node[left]{$A_{n+1}((y_i)_i)$};
 \draw[red] (1.6,1.8) node[left]{$A_n((y_i)_i)$};
 \fill[red] (1.5,1.75)--(1.7,1.75)--(1.7,1.85)--(1.5,1.85)--(1.5,1.75);
  \draw[red] (4.2,1.6)--(4.2,1.4)--(4.6,1.4)node[pos=0.82,right]{\scriptsize $>m$}--(4.6,1.6);
 \draw[<->,white!60!red] (4,1.8)--(4,2.4);
 \draw[red] (4,2.1)node[right]{\scriptsize$m+1$};
    \end{scope}
 \end{tikzpicture}
 \caption{}\end{subfigure}
  \begin{subfigure}[b]{0.4\textwidth}
\centering
\begin{tikzpicture}
 \draw[pattern=my crosshatch dots]  (0,1) rectangle (5.5,5);
 \draw[lightgray] (0,1.8)--(5.5,1.8);
 \fill[blue] (4,1.75)--(4.2,1.75)--(4.2,1.85)--(4,1.85)--(4,1.75);
 \draw[blue] (4.1,1.85) node[below]{$A_{j+1}((x_i)_i)$};
 \fill[blue] (1.5,1.75)--(1.7,1.75)--(1.7,1.85)--(1.5,1.85)--(1.5,1.75);
\draw[blue] (1.6,1.85) node[below]{$A_{j}((x_i)_i)$};
 \begin{scope}[rotate=-30,shift={(-1.8,2)}]
   \draw[white!60!red] (4,1.8) circle (17pt);
     \draw[lightgray] (-0.39,1.8)--(4,1.8);
 \fill[red] (4,1.75)--(4.2,1.75)--(4.2,1.85)--(4,1.85)--(4,1.75);
  \draw[red] (4.1,1.8) node[left]{$A_{n+1}((y_i)_i)$};
 \draw[red] (1.6,1.8) node[left]{$A_n((y_i)_i)$};
 \fill[red] (1.5,1.75)--(1.7,1.75)--(1.7,1.85)--(1.5,1.85)--(1.5,1.75); 
 \end{scope}
 \end{tikzpicture}
  \caption{}\end{subfigure} 
  
  \vspace{0.3cm}
  
  \begin{subfigure}[b]{0.4\textwidth}
\centering
\begin{tikzpicture}
 \draw[pattern=my crosshatch dots]  (0,1) rectangle (5.5,5);
  \draw[lightgray] (0,1.8)--(4,1.8);
 \fill[blue] (4,1.75)--(4.2,1.75)--(4.2,1.85)--(4,1.85)--(4,1.75);
 \draw[blue] (4.1,1.85) node[below]{$A_{n+1}((x_i)_i)$};
 \fill[blue] (1.5,1.75)--(1.7,1.75)--(1.7,1.85)--(1.5,1.85)--(1.5,1.75);
\draw[blue] (1.6,1.85) node[below]{$A_{n}((x_i)_i)$};
 \begin{scope}[rotate=-30,shift={(-1.8,2)}]
   \draw[white!60!red] (4,1.8) circle (17pt);
      \draw[white!60!red] (1.5,1.8) circle (17pt);
     \draw[lightgray] (-0.39,1.8)--(4,1.8);
 \fill[red] (4,1.75)--(4.2,1.75)--(4.2,1.85)--(4,1.85)--(4,1.75);
  \draw[red] (4.1,1.8) node[left]{$A_{n+1}((y_i)_i)$};
 \draw[red] (1.6,1.8) node[left]{$A_n((y_i)_i)$};
 \fill[red] (1.5,1.75)--(1.7,1.75)--(1.7,1.85)--(1.5,1.85)--(1.5,1.75); 
 \end{scope}
 \end{tikzpicture}
  \caption{}
  \end{subfigure}
  \begin{subfigure}[b]{0.4\textwidth}
\centering
\begin{tikzpicture}
 \draw[pattern=my crosshatch dots]  (0,1) rectangle (5.5,5);
   \draw[lightgray] (0,1.8)--(4,1.8);
 \fill[blue] (4,1.75)--(4.2,1.75)--(4.2,1.85)--(4,1.85)--(4,1.75);
 \draw[blue] (4.1,1.85) node[below]{$A_{n+1}((x_i)_i)$};
 \fill[blue] (1.5,1.75)--(1.7,1.75)--(1.7,1.85)--(1.5,1.85)--(1.5,1.75);
\draw[blue] (1.6,1.85) node[below]{$A_{n}((x_i)_i)$};
 \begin{scope}[shift={(-0.2,0.35)}]
   \draw[white!60!red] (4,1.8) circle (17pt);
      \draw[white!60!red] (1.5,1.8) circle (17pt);
     \draw[lightgray] (0.2,1.8)--(4,1.8);
 \fill[red] (4,1.75)--(4.2,1.75)--(4.2,1.85)--(4,1.85)--(4,1.75);
  \draw[red] (4.1,1.8) node[above]{$A_{n+1}((y_i)_i)$};
 \draw[red] (1.6,1.8) node[above]{$A_n((y_i)_i)$};
 \fill[red] (1.5,1.75)--(1.7,1.75)--(1.7,1.85)--(1.5,1.85)--(1.5,1.75); 
 \end{scope}
 \end{tikzpicture}
  \caption{}
  \end{subfigure}

    \caption{A representation of the steps of the part of the proof of Lemma \ref{lemma:PPS^R_ce} showing that we can assume, without loss of generality, $\Phi$ to be a translation in the case $d=2$. A portion of the periodic point set $\psi_m^n((x_i)_i)$ is represented by the points contained in the blue boxes and the black dots belonging to $X$. We represent in red the images of the blocks $A_n((y_i)_i)$ and $A_{n+1}((y_i)_i)$ along the isometry $\Phi$. The red ball of radius $m+1$ around the first of the $K$ points of $A_{n+1}((y_i)_i)$ (which contains the union of the balls of radius $m$ centered in the $K$ points of $A_{n+1}((y_i)_i)$) does not contain $L\leq K$ points of $\psi_m^n((x_i)_i)$ if it does not intersect any of the blue blocks (see \eqref{eq:L_def}) This situation is shown in (a). Hence, any matching between the two periodic point sets connects at least a pair of points that are more than $m$ far apart. The same happens if the red ball intersects a block $A_{j+1}((x_i)_i)$ with $j\neq n$, as in (b). In fact, the blue block can absorb at most $L$ points, so there will still be at least $L$ unmatched points in $A_{n+1}((y_i)_i)$. In (c), the ball of radius $m+1$ centred in the first of the points of the red block $A_n((y_i)_i)$ does not intersect any blue block, and so again the conclusion since that red block contains $L$ points. Since these observations can be repeated for every red block, $\Phi$ does not rotate $\psi_m^n((y_i)_i)$ as depicted in (d).}\label{fig:lower_bound}
\end{figure}
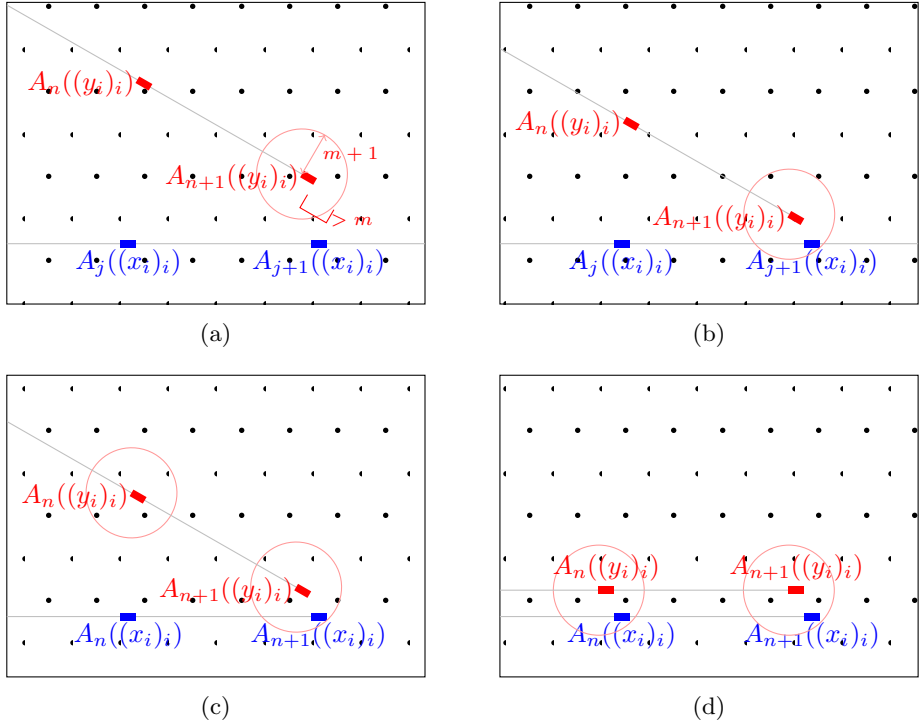

If $k=0$, then Claim \ref{claim:matching} implies that for every $j\in\{1,\dots,n\}$ there exists $x\in A_j((x_i)_i)$ such that $$\lvert\lvert x-f(x)\rvert\rvert\geq\lvert x_i-y_i\rvert-\varepsilon$$ (if $x_i<y_i$, the right-most point of $A_j((x_i)_i)$ can be matched with the left-most one of $A_j((y_i)_i)$). If $k$ were not $0$, this distance could decrease. However, it would also create a mismatch between the points of $A_{n+1}((x_i)_i)$ and $A_{n+1}((y_i)_i$. More precisely, according to Claim \ref{claim:matching_n+1},
$$\max_{x\in A_{n+1}((x_i)_i)}\lvert\lvert x-f(x)\rvert\rvert\geq\dist(A_{n+1}((x_i)_i),ke_1+A_{n+1}((y_i)_i))\geq k-\varepsilon.$$
Thus, the best trade off can achieve at most the value $((d_m^n((x_i)_i,(y_i)_i)-\varepsilon)-\varepsilon)/2$, which is the lower bound in \eqref{eq:EB_PPS^R}.
\end{proof}

We are now ready to prove the second part of Theorem \ref{theo:no_coarse_embedding}.

\begin{proof}[Proof of Theorem \ref{theo:no_coarse_embedding} for $\PPS^R(\R^d,\kappa)$] Let us adopt the notation above. Similarly to the case of $\PPS_r(\R^d,\kappa)$, we fix the parameters defining the embeddings $\psi_m^n$ recursively on the order $\preceq$. Let us consider a pair of indices $(m,n)\in(\N\setminus\{0\})^2$, and suppose that we have defined the embeddings $\psi_{m^\prime}^{n^\prime}$ for every $(m^\prime,n^\prime)\prec(m,n)$. We require that
\begin{equation}\label{eq:def_K_PPSR}K\geq\max_{\substack{x\in\R^d\\(m^\prime,n^\prime)\prec(m,n)\\
(x_i)_i\in C_{m^\prime}^{n^\prime}}}\lvert B(x,2^{T(m,n)}+1)\cap\psi_{m^\prime}^{n^\prime}((x_i)_i)\rvert +1.\end{equation}
Then, with this further requirement, we can choose $l$ and $K$ to match the desired density. Thus, $\{\psi_m^n\}_{m,n\in\N\setminus\{0\}}$ is a coarse embedding into $\PPS^R(\R^d,\kappa)$. We just need to show that 

$$\dist_{EB}(\psi_m^n(C_m^n),\psi_{m^\prime}^{n^\prime}(C_{m^\prime}^{n^\prime}))\to\infty$$ to apply Lemma \ref{lemma:emb_from_X_to_cube} and Theorem \ref{theo:Laf} and conclude the proof.

Let $(m,n),(m^\prime,n^\prime)\in(\N\setminus\{0\})^2$ be two distinct pair of indices. Without loss of generality, we can assume that $(m^\prime,n^\prime)\prec(m,n)$. Then, thanks to \eqref{eq:def_K_PPSR} and Lemma \ref{lemma:too_many_points}, for every $(x_i)_i\in C_m^n$, $(y_i)_i\in C_{m^\prime}^{n^\prime}$, $\Phi\in\Iso(\R^d)$ and every bijection $f\colon\psi_m^n((x_i)_i)\to\Phi(\psi_{m^\prime}^{n^\prime}((y_i)_i))$, there exists $x\in A_{n+1}((x_i)_i,m,n)$ such that $\lvert\lvert x-f(x)\rvert\rvert>2^{T(m,n)}$. Hence, $\EB(\psi_m^n((x_i)_i),\psi_{m^\prime}^{n^\prime}((y_i)_i))>2^{T(m,n)}$, and so $$\dist_{EB}(\psi_m^n(C_m^n),\psi_{m^\prime}^{n^\prime}(C_{m^\prime}^{n^\prime}))>2^{T(m,n)}\to\infty.$$
\end{proof}

\appendix

\section{On the definition of the extended pseudo-metrics}\label{appendix}

\begin{proposition}
\label{PropDbMIN}
Assume $X$ and $Y$ are periodic point sets in $\R^n$. In the definition of $d_B$ (Definition \ref{DefDB})
$$
d_B(X,Y)=\inf_{f\colon X\to Y\text{ bijection}}\sup_{x\in X}\lvert\lvert x-f(x)\lvert\lvert
$$
the bijection realizing the infimum is attained (i.e., we can replace inf by min), while the supremum may not be attained (i.e., we cannot replace sup by max).
\end{proposition}

\begin{proof}
In order to prove that the supremum may not be attained consider two lattices in $\R^2$, lattice $X$ generated by $(1,0)$ and $(\alpha,1)$ for some irrational $\alpha \in (0,1)$, and lattice $Y$ being the standard $\Z^2$ lattice. 

Observe that the vertical distance between horizontal cosets is $1$. Thus the coset shift along the horizontal cosets (see Figure \ref{fig:coset}) is a bijection $f \colon X \to Y$ mapping each point of $X$ to its unique (by irrationality of $\alpha$) closest point in $Y$ at distance less than $1/2$, and is thus an optimal bijection. (See Lemma \ref{LemDifference} for a similar argument for the case of $d_{EB}$) On the other hand, as $\alpha$ is irrational we have $\sup_{x\in X}\lvert\lvert x-f(x)\lvert\lvert=1/2$.

We now prove that the bijection realizing the infimum is attained. Assume $d_B(X,Y)=r$. Then for each $i \in \N$ there exists a bijection $f_i \colon X \to Y$ such that $d(x,f_i(x))\leq r + 1/i$. Let $\Phi = \{f_i\}_{i\in \N}$ denote the resulting sequence of bijections. We will construct a ``limit bijection'' $f$. 

Let $0$ denote the zero vector in $\R^n$.
\begin{enumerate}
    \item For each $i$ we have $f_i(B(0,r)\cap X)\subset B(0,3r)\cap Y$. As $B(0,r)\cap X$ and $B(0,3r)\cap Y$ are finite we may replace $\Phi$ by a subsequence such that functions $f_i$ coincide on $B(0,r)\cap X$. Define $f|_{B(0,r)\cap X}=f_i$.
    \item For each $i$ we have $f^{-1}_i(B(0,3r)\cap Y)\subset B(0,5r)\cap X$. As before we may replace $\Phi$ by a subsequence such that functions $f^{-1}_i$ coincide on $B(0,3r)\cap Y$. 
    \item We proceed inductively by alternating between steps 1. and 2. to progressively define $f$ on an ever larger set. For example, in the third step we  replace $\Phi$ by a subsequence such that functions $f_i$ coincide on $B(0,5r)\cap X$, and define $f|_{B(0,5r)\cap X}=f_i$.
\end{enumerate}
The resulting rules defining $f$ on progressively larger balls constitute a well defined bijection:
\begin{itemize}
    \item Map $f$ is well defined: For each $x \in X$ point $f(x)$ is defined when defining $f|_{B(0,(2k+1)r)\cap X}$ for the first $k$ for which $(2k+1)r> ||x||$ and the definition does not change during later steps.
    \item Map $f$ is injective: For points $x_1 \neq x_2$ of $X$ contained in some $B(0,(2k+1)r)$, their $f$ images differ as the definition of $f|_{B(0,(2k+1)r)\cap X}$ in an inductive step arises from injective function.
    \item Map $f$ is surjective: Given $y\in Y \cap B(0,4kr)$, an inductive version of step 2. guarantees that for all subsequently considered bijections $f_i$ the preimage $f_i^{-1}(y)$ is a fixed single point, whose image via $f$ is $y$.
\end{itemize}
By construction $d(x,f(x)) \leq r, \forall x\in X$ as $f(x)=f_i(x)$ for infinitely many original functions $f_i$. Thus $f$ is a bijection at which the infimum is attained.
\end{proof}

\begin{proposition}
\label{PropDebMIN}
Assume $X$ and $Y$ are periodic point sets in $\R^n$. In the definition of $\EB$ (Definition \ref{DefDB})
$$
\EB(X,Y)=\inf_{\psi\in\Iso(\R^d)}d_B(X,\psi(Y)),
$$
the isometry realizing the infimum is attained (i.e., we can replace inf by min).
\end{proposition}

\begin{proof}
Assume $\EB(X,Y)=r$. Then for each $i \in \N$ there exist a bijection $f_i \colon X \to Y$ and $\psi_i \in \Iso(\R^n)$ such that $d(x,\psi_i(f_i(x)))\leq r + 1/i, \forall x \in X$. Each isometry $\psi_i$ decomposes uniquely as $\psi_i = \rho_i \circ \tau_i$, where $\tau_i$ is a translation and $\rho_i$ is a basepoint preserving isometry, i.e., $\rho_i \in O(n)$. As translations by the generators of the underlying lattice $\Lambda$ of $Y$ do not change $Y$, we may assume $\tau_i$ is a member of a chosen fixed unit cell of $\Lambda$. The closure of the unit cell and $O(n)$ are compact, thus by taking a subsequence we may assume $\tau_i$ converge towards translation $\tau$ by an element of the closure of the unit cell, and $\rho_i$ converge towards $\rho \in O(n)$. We claim $\psi = \rho \circ \tau$ is an isometry at which the infimum in the definition of $\EB(X,Y)$ is attained, i.e., $d_B(X,\psi(Y))=r$.

By the same argument as in the previous proposition we can construct a bijection $f\colon X 
\to Y$ such that for each $x\in X$ we have  $f(x)=f_i(x)$ for infinitely many indices $i$. It remains to show that $d(x,\psi(f(x)))\leq r$ for each $x\in X$. Let $\varepsilon > 0$ and fix $x \in X$. Choose large enough $i$ so that:
\begin{enumerate}
    \item $\varepsilon > 2/i$.
    \item $f(x)=f_i(x)$.
    \item $d(\psi_i(f(x)),\psi(f(x)))< \varepsilon/2$.
\end{enumerate}
Then 
$$
d(x,\psi(f(x))) \leq d(x,\psi_i(f(x))) + d(\psi_i(f(x)),\psi(f(x))) \leq r + 1/i +\varepsilon/2 = r + \varepsilon.
$$
As $\varepsilon>0$ is arbitrary we obtain $d(x,\psi(f(x)))\leq r, \forall x\in X$.
\end{proof}

\Addresses


\begin{thebibliography}{10}

\bibitem{Ant1} S. A. Antonyan, {\em The Gromov-Hausdorff hyperspace of a Euclidean space}, Advances in Mathematics, Volume 363, 2020, 106977, ISSN 0001-8708, {\tt  https://doi.org/10.1016/j.aim.2020.106977}.
\bibitem{Ant2} S. A. Antonyan, {\em The Gromov-Hausdorff hyperspace of a Euclidean space, II}, Advances in Mathematics, Volume 393, 2021, 108055, ISSN 0001-8708, {\tt https://doi.org/10.1016/j.aim.2021.108055}.


\bibitem{BatGar} D. Bate, A. L. Garcia-Pulido, {\em Bi-Lipschitz embeddings of the space of unordered m-tuples with a partial transportation metric}, {\tt arXiv:2212.01280}.

\bibitem{Bie1} L. Bieberbach, {\"Uber die Bewegungsgruppen der Euklidischen R\"aume I}, Mathematische Annalen, Volume 70
1911, 297--336, {\tt https://doi.org/10.1007/BF01564500}.

\bibitem{Bie2} L. Bieberbach, {\"Uber die Bewegungsgruppen der Euklidischen R\"aume II}, Mathematische Annalen, Volume 72
1912, 400--412, {\tt https://doi.org/10.1007/BF01456724}.


\bibitem{Blu} L. Blumenthal, {\em Theory and Applications of Distance Geometry}, At the Clarendon Press, Oxford, 1953.

\bibitem{Bog} O. V. Bogopolski, {\it Infinite commensurable hyperbolic groups are bi-Lipschitz equivalent}, Algebra and Logic, Volume 36, No. 3, 1997, 155--163.

\bibitem{BubWag} P. Bubenik, A. Wagner, {\em Embeddings of persistence diagrams into Hilbert spaces}, J. Appl. Comput. Topol. 4 (2020), no. 3, 339--351, DOI 10.1007/s41468-020-00056-w.
MR4130975.
\bibitem{CarBau} M. Carri\`ere, U. Bauer, {\em On the metric distorsion of embedding persistence diagrams into separable Hilbert spaces} in 35th International Symposium on Computational Geometry. Ed. Gill Barequet and Yusu Wang. LIPIcs. Leibniz International Proceedings in Informatics, 129. Wadern, Germany: Dagstuhl Publishing, 2019. 21:1--21:15.

\bibitem{CKMRV17}
H. Cohn, A. Kumar, S. Miller, D. Radchenko, M. Viazovska, {\it The sphere packing problem in dimension 24}, Ann. of Math. {\bf 185}:3 (2017), 1017--1033.

\bibitem{CDHT01} J. H. Conway, O. Delgado Friedrichs, D. T. Huson, W. P. Thurston, {\it On Three-Dimensional Space Groups}, Beitr\"age zur Algebra und Geometrie, Volume 42, No 2, 2001, 475--507.


\bibitem{DunOgu} M. Duneau, C. Oguey, {\em Bounded interpolations between lattices}, 1991 J. Phys. A: Math. Gen. 24 461.

\bibitem{DraGonLafYu} A. N. Dranishnikov, G. Gong, V. Lafforgue, and G. Yu, {\em Uniform embeddings into Hilbert space and a question of Gromov}, Canad. Math. Bull. 45 (2002), no. 1, 60--70, DOI 10.4153/CMB2002-006-9. MR1884134.


\bibitem{EdeHeiKurSmiWin} H. Edelsbrunner, T. Heiss, V. Kurlin, P. Smith, M. Wintraecken, {\em The density
fingerprint of a periodic point set}. In: Proceedings of SoCG (2021).
\bibitem{Edw} D. A. Edwards, {\em The Structure of Superspace}, published in: Studies in Topology, Academic Press, 1975.

\bibitem{FG-pisot} D. Frettl\"oh, A. Garber, {\it Pisot substitution sequences, one dimensional cut-and-project sets and bounded remainder sets with fractal boundary}, Indagationes Mathematicae, Volume 29, Issue 4, 2018, 1114--1130.

\bibitem{FG-weighted} D. Frettl\"oh, A. Garber, {\it Weighted $1\times 1$  Cut-and-Project Sets in Bounded Distance to a Lattice}, Discrete \& Computational Geometry, Volume 62, 2019, 649--661.

\bibitem{fre22}
D. Frettl\"oh, A. Garber, L. Sadun, {\em Number of bounded distance equivalence classes in hulls of repetitive Delone sets}, Disc. \& Cont. Dyn. Sys. 42 (2022), no. 3, 1403--1414.

\bibitem{Gro_GH_fr} M. Gromov, {\em Structures M\'etriques pour les Vari\'et\'es Riemanniennes} (J. Lafontaine and P. Pansu, eds.), Textes Math\'ematiques, 1, CEDIC, Paris, 1981.

\bibitem{Grom}
M. Gromov, {\em Asymptotic invariants for infinite groups}, in
Geometric Group Theory, vol. 2, 1--295, G. Niblo and M. Roller,
eds., Cambridge University Press, 1993.

\bibitem{Hal05}
T. Hales, {\it A proof of the Kepler conjecture}, Ann. of Math. {\bf 162} (2005), 1065--1185.

\bibitem{Haynes} A. Haynes, {\it Equivalence classes of codimension-one cut-and-project nets}, Ergodic Theory and Dynamical Systems, Volume 36, Issue 3, 2016, 816--831, {\tt https://doi.org/10.1017/etds.2014.90}.

\bibitem{HayKoi} A. Haynes, H. Koivusalo, {\it Constructing bounded remainder sets and cut-and-project sets which are bounded distance to lattices}, Isr. J. Math., Volume 212, Issue 1, 2015, 189--201.

\bibitem{HKK} A. Haynes, M. Kelly, H. Koivusalo, {\it  Constructing bounded remainder sets and cut-and-project sets which are bounded distance to lattices, II}, Indag. Math., Volume 28, Issue 1, 2017, 138--144.

\bibitem{Isb}  J.R. Isbell, {\em Uniform Spaces}, American Mathematical Society, 1964.

\bibitem{Lac} M. Laczkovich, {\it Uniformly spread discrete sets in $\mathbb R^d$}, J. Lond. Math. Soc., Volume 46, No. 1, 1992, 39--57.


\bibitem{Laf} V. Lafforgue, {\em Un renforcement de la propri\'{e}t\'{e} (T)}, Duke Math. J. 143 (2008), no. 3, 559--602. 
\bibitem{MajVitWen} S. Majhi, J. Vitter, C. Wenk, {\em Approximating Gromov-Hausdorff Distance in Euclidean Space}, preprint, {\tt arXiv:1912.13008v3}.

\bibitem{ManMar22} Y. Manin, M. Marcoli, {\it Computability questions in the sphere packing problem}, preprint, 2022, \url{https://arxiv.org/abs/2212.05119v1}.

\bibitem{Mem07} F. M\'emoli, {\em On the use of Gromov-Hausdorff distances for shape comparison}, Point Based Graphics, 2 (2007).

\bibitem{Mem} F. M\'emoli, {\em Gromov-Hausdorff distances in Euclidean spaces}, in ``2008 IEEE Computer Society Conference on Computer Vision and Pattern Recognition Workshops'', (Anchorage, AK, USA), pp. 1--8, IEEE, June 2008.

\bibitem{MemSap1} F. M\'emoli, G. Sapiro, {\em Comparing point clouds}, in: SGP '04: Proceedings of the 2004 Eurographics/ACM SIGGRAPH Symposium on Geometry Processing, pp. 32--40. ACM, New York (2004).
	\bibitem{MemSap2} F. M\'emoli, G. Sapiro, {\em A theoretical and computational framework for isometry invariant recognition	of point cloud data}, Found. Comput. Math. 5(3), 313--347 (2005).
\bibitem{MitVir} A. Mitra, \v{Z}. Virk, {\em The space of persistence diagrams of $n$ points coarsely embeds into Hilbert space}, Proc. Amer. Math. Soc. 149 (2021), 2693--2703.

\bibitem{Rad} C. Radin, {\it The open mathematics of crystallization}, Notices AMS, Volume 64, No. 6, 2017, 551--556.


\bibitem{Roe} J. Roe, \emph{Lectures on Coarse Geometry}, 
University lecture Series, Am. Math. Soc., 2003.

\bibitem{Sen} M. Senechal, {\it Quasicrystals and Geometry}, Cambridge University Press, Cambridge, 1995.

\bibitem{Sch08} A. Sch\"urmann, {\it Computational Geometry of Positive Definite Quadratic Forms}, University lecture Series, Am. Math. Soc., 2008.

\bibitem{Sol} Y. Solomon, {\it A simple condition for bounded displacement}, J. Math. Anal. Appl., Volume 414, 2014, 134--148.

\bibitem{Via17}
M. Viazovska, {\it The sphere packing problem in dimension 8}, Ann. of Math. {\bf 185}:3 (2017), 991--1015.


\bibitem{WeiYamZav} T. Weighill, T. Yamauchi, N. Zava, {\em Coarse infinite-dimensionality of hyperspaces of finite subsets}, Eur. J. Math., to appear.

\bibitem{Yu} G. Yu, {\em The coarse Baum-Connes conjecture for spaces which admit a uniform embedding into Hilbert space}, Invent. Math. 139 (2000), no. 1, 201--240, DOI
10.1007/s002229900032. MR1728880

\bibitem{Za} N. Zava, {\em Cowellpoweredness and closure operators in categories of coarse spaces}, Topology Appl. 268 (2019), 106899.

\bibitem{Za_GH} N. Zava, {\em Non-embeddability results of the Gromov-Hausdorff space}, preprint.

\end{thebibliography}
\end{document}